\documentclass[a4paper,12pt]{article}
\usepackage[latin1]{inputenc}
\usepackage{amsfonts}
\usepackage{amsthm}
\usepackage{amsmath}
\usepackage{amsfonts}
\usepackage{latexsym}
\usepackage{amssymb}
\usepackage{dsfont}
\usepackage[usenames]{color}

\newtheorem{teo}{Theorem}[section]
\newtheorem*{teo*}{Theorem}
\newtheorem{lem}[teo]{Lemma}
\newtheorem{cor}[teo]{Corollary}
\newtheorem{pro}[teo]{Proposition}

\theoremstyle{definition}
\newtheorem{fed}[teo]{Definition}
\theoremstyle{remark}
\newtheorem{exa}[teo]{Example}
\newtheorem{rem}[teo]{Remark}

\newtheorem{exas}[teo]{Examples}

\def\coma{\, , \, }

\def\py{\peso{and}}
\newcommand{\peso}[1]{ \quad \text{ #1 } \quad }

\def\n0{n_{ \text{\rm \tiny o}}}

\def\bce{\begin{center}}
\def\ece{\end{center}}

\def\cO{{\mathcal O}}
\def\cD{\mathcal D}
\def\cJ{{\mathcal J}}
\def\cI{{\mathcal I}}

\def\noi{\noindent}
\def\cF{\mathcal F}

\def\bm{\left[\begin{array}}
\def\em{\end{array}\right]}
\def\ben{\begin{enumerate}}
\def\een{\end{enumerate}}
\def\bit{\begin{itemize}}
\def\eit{\end{itemize}}
\def\barr{\begin{array}}
\def\earr{\end{array}}

\def\e{\mathrm{e}}

\def\la{\lambda}

\def\al{\alpha}
\def\N{\mathbb{N}}
\def\R{\mathbb{R}}
\def\C{\mathbb{C}}

\def\Z{\mathbb{Z}}

\def\cJ{\mathcal{J}}
\def\cD{\mathcal{D}}

\def\cC{\mathcal{C}}

\def\cE{\mathcal{E}}

\def\cH{\mathcal{H}}
\def\cK{\mathcal{K}}

\def\cP{\mathcal{P}}

\def\cM{{\cal M}}
\def\cB{{\cal B}}
\def\cN{{\cal N}}
\def\cV{{\cal V}}
\def\cU{{\cal U}}

\def\fS{\mathfrak{S}}

\def\da{^\downarrow}

\DeclareMathOperator{\Tr}{Tr} \DeclareMathOperator{\tr}{Tr}

\def\beq{\begin{equation}}
\def\eeq{\end{equation}}

\def\Ax2{\,( S_{E(\cF)^\#_\cV})\hat{}_x }

\newcommand{\PI}[2]{\left\langle #1 , #2 \right\rangle}
\newcommand{\SE}[2]{\{ #1_{#2} \}_{#2 \geq 1}}

\begin{document}
\title{On restricted diagonalization}
\author{Eduardo Chiumiento and Pedro Massey} 
\date{}
\maketitle

\begin{abstract}
Let $\cH$ be a separable infinite-dimensional complex  Hilbert space, $\cB(\cH)$ the algebra of bounded linear operators acting on $\cH$ and $\cJ$  a  proper two-sided ideal of $\cB(\cH)$. Denote  by $\cU_\cJ(\cH)$  the group  of all unitary operators  of the form $I+\cJ$. 
Recall that an operator $A \in \cB(\cH)$ is diagonalizable if there exists a unitary operator $U$ such that $UAU^*$ is diagonal with respect to some orthonormal basis. A more restrictive notion of diagonalization can be formulated with respect to a fixed  orthonormal basis $\e=\{ e_n\}_{n\geq 1}$ and a proper operator ideal $\cJ$ as follows:   $A \in \cB(\cH)$ is called restricted diagonalizable  if there exists 
$U\in \cU_\cJ(\cH)$ such that $UAU^*$ is diagonal with respect to $\e$. 
In this work we give  necessary and sufficient conditions for a diagonalizable operator to be restricted diagonalizable. Our conditions become a characterization of those diagonalizable operators which are restricted diagonalizable when the ideal is arithmetic mean closed. Then we obtain  results on the structure of the set of all  restricted diagonalizable operators. In this way 
we answer several open problems recently 
raised by Belti\c{t}$\breve{\text{a}}$,  Patnaik and  Weiss.
\end{abstract}

\bigskip



\section{Introduction}

Let $\cH$ be a separable infinite-dimensional complex  Hilbert space. An \textit{operator ideal} $\cJ$ is a two-sided ideal of the algebra of bounded linear operators $\cB(\cH)$. Each operator ideal has associated  a subgroup of the unitary operators defined by $\cU_\cJ(\cH):=\cU(\cH) \cap (I+ \cJ)$, where $\cU(\cH)$ denotes the unitary group of $\cH$. Throughout, we fix an orthonormal basis $\e=\{ e_n\}_{n\geq 1}$ of $\cH$
and we let $\cD$ denote the corresponding algebra of diagonal operators.
 We call an operator $A\in \cB(\cH)$ \textit{diagonalizable} if there exists a unitary $U\in\cU(\cH)$ such that $UAU^*\in\cD$.
 In their study of Lie theoretic properties of operator ideals, Belti\c{t}$\breve{\text{a}}$, Patnaik and Weiss introduced in \cite{BPW16} the notion of \textit{$\cU_\cJ(\cH)$-diagonalizable operators with respect to $\e$}, or simply, \textit{restricted diagonalizable operators}, which refers to those diagonalizable operators  
that can be diagonalized by a unitary $U \in \cU_\cJ(\cH)$. 
Furthermore, given two operator  ideals $\cI$, $\cJ$, they 
offered the following versions of restricted diagonalization defined by the sets:
$$
\cD_{\cJ,\cI}:=\{ UDU^* : D \in \cD \cap \cI, \, U \in \cU_\cJ(\cH) \} \subset\cI.
$$ 
For short, we write 
$\cD_\cJ:=\cD_{\cJ,\cJ}$.   In particular, the set of all 
restricted diagonalizable operators is given by $\cD_{\cJ,\cB(\cH)}$. In contrast to the notion of diagonalization, which 
is independent of the orthonormal basis, we point out that when the operator ideal $\cJ$ is proper (i.e. $\{0\}\neq \cJ\neq \cB(\cH)$), the choice of the orthonormal basis $\e$ plays a central role to determine whether an operator is restricted diagonalizable.

The aim of this paper is twofold. On the one hand, we give  necessary and 
sufficient conditions for an arbitrary diagonalizable operator  to be  $\cU_\cJ(\cH)$-diagonalizable with respect to $\e$ (for a detailed description of these conditions see Eq. \eqref{eq nuestras condiciones intro} below). On the other hand,
we study the structure of the self-adjoint part of $\cD_{\cJ}$ and the way it sits inside the self-adjoint part of the ideal $\cJ$. These  results allow us to answer  open problems stated in \cite{BPW16}.

In order to put our results in perspective, we describe some of the contexts in which restricted diagonalization has been considered. As a general framework, we mention the theory of Banach-Lie groups, which deals with  infinite dimensional Lie groups modeled on Banach spaces. These groups have been useful to encode, in geometrical terms, several structures arising in operator theory and its applications, and this has led to a fruitful interaction between functional analytic and geometric techniques. The groups $\cU_\cJ(\cH)$ are indeed Banach-Lie groups, whenever the operator ideal $\cJ$ admits a complete norm stronger than the operator norm (see, e.g. \cite{B, neeb0}). For instance, 
the well-known Schatten ideals $\fS_p(\cH)$ ($1\leq p \leq \infty$) satisfy this condition. 
The book \cite{dlHar72} by P. de la Harpe  is a standard reference on these types of groups. Other references about Banach-Lie groups associated to operator ideals and their homogeneous spaces are related to structure and representation theory of Banach-Lie groups \cite{B09, BPW16, neeb98, neeb0}, Banach Lie-Poisson spaces  \cite{BR05}, infinite dimensional Kh\"{a}ler geometry \cite{T07} and metric geometry of homogeneous spaces   \cite{AL8, ALR10, La19}. We also refer to the books \cite{B, U85} and the references therein for more information.

Motivated by the role of Cartan subalgebras in the structure theory of finite dimensional Lie algebras, the authors in  \cite{BPW16} studied  these subalgebras in the 
setting of operator ideals on an infinite dimensional separable Hilbert space. Cartan subalgebras were introduced there  as maximal abelian self-adjoint subalgebras of operator ideals. 
 When the operator ideal $\cJ$ is proper, the action of the group $\cU_\cJ(\cH)$ on the set of Cartan subalgebras  defines smaller orbits than the orbits provided by the full unitary group $\cU(\cH)$.  Among other results, they showed that  the action of $\cU_\cJ(\cH)$ on Cartan subalgebras has uncountable many distinct orbits. 
The notion of restricted diagonalization was introduced there, and used in deriving this and other results.   

 Additionally, we remark that restricted diagonalization had previously appeared in the literature in some specific cases; see  Hinkkanen  \cite{H85} for a sufficient condition to  belong to $\cD_{\fS_2(\cH)}$ (Hilbert-Schmidt operators case), and also the results about unitary equivalence of projections contained in the works by Brown, Douglas and Fillmore \cite{BDF73}, Str$\breve{\text{a}}$til$\breve{\text{a}}$ and Voiculescu \cite{SV78},  Carey \cite{C85}, and Kaftal and Loreaux \cite{KL17}. They respectively treated the cases of when two orthogonal projections are unitary equivalent by means of a unitary in $\cU_\cJ(\cH)$  for the ideals $\cJ=\cK(\cH)$ (compact operators), $\cJ=\fS_2(\cH)$ (Hilbert-Schmidt operators), $\cJ$ a symmetrically-normed ideal, and  an arbitrary proper operator ideal. 
We point out that these results rely on the  notion of \textit{essential codimension of projections} (see also  \cite{AS94, ASS94}), and more
recently, were  reinterpreted in the light of restricted diagonalization  by Loreaux \cite{L19} to give a characterization of  $\cU_\cJ(\cH)$-diagonalization of  projections. 

This, in turn,  allowed  the characterization
of those finite spectrum normal operators that are  $\cU_\cJ(\cH)$-diagonalizable provided by Loreaux in the same work \cite{L19}. 
It is interesting to notice that the motivations   for this characterization are related with Arveson's index obstruction \cite{Arv07, KL17} satisfied by the diagonals of certain normal operators with finite spectrum, which is well known to be a generalization of Kadison's integer condition \cite{Kad02}  for the diagonals of self-adjoint projections (for related work on diagonals of operators and index obstructions see \cite{ArvKad06, BJ15, KW10, LW16, MT19}). 
Moreover, following the line of previous research on these topics, Loreaux showed the deep relations between the study of several questions raised in \cite{BPW16} and the work  \cite{BDF73} on unitary equivalence modulo compact operators.

\medskip

The contents of this paper are as follows. In Section \ref{Prelim} we present the necessary background on operator ideals on Hilbert spaces,  essential codimension of projections and  restricted diagonalization.

In Section \ref{restricted diagonalization},
we give  necessary and sufficient conditions for an arbitrary (normal) diagonalizable operator $A$ to be  $\cU_\cJ(\cH)$-diagonalizable with respect the orthonormal basis $\e$ (Theorem \ref{main result}). Let $\{P_n\}_{n=1}^N $ ($N\in\N$ or $N=\infty$) denote the spectral projections of $A$ associated to its point spectrum (eigenvalues). Then our sufficient condition for restricted diagonalization can be stated as follows: there exists a family $\{E_n\}_{n=1}^N$ of projections in the diagonal algebra $\cD$ such that 
\beq\label{eq nuestras condiciones intro}
 \sum_{n = 1}^N  (I-E_n) P_n \in \cJ \, \, \, \, \text{  and  } \, \, \, \,  \sum_{n =1}^N E_n(I- P_n ) \in \cJ,
\eeq 
where both series are assumed to converge in the weak operator topology  when $N=\infty$. 

We observe   that this condition is not merely analytic,   it is also geometric in nature. Indeed, it allows us to define part of the unitary operator that diagonalizes the operator $A$ by using a technique to construct local cross sections in many examples of homogeneous spaces of unitary groups (see, e.g. \cite{ACR19, AL8, ED13}). Usually this technique is adapted to each specific example of homogeneous space, though it is based on the well-known fact (which goes back at least to \cite{RN}) that two orthogonal projections lying at distance less than one must be unitary equivalent. We also remark here that an extension of  previous work \cite{EC10} on  $\cU_\cJ(\cH)$-orbits of partial isometries and the essential codimension plays a key role in the proof of our sufficient condition.

In case $N\in\N$ the condition in Eq. (\ref{eq nuestras condiciones intro})  is also necessary, so it complements  Loreaux's characterization for $\cU_\cJ(\cH)$-diagonalization of finite spectrum operators (see Theorem \ref{teo1 L19}). 
Furthermore, we  show that under the assumption that $\cJ$ is an \textit{arithmetic mean closed ideal} (see \cite{Dyk04, KW11}), the condition in Eq. \eqref{eq nuestras condiciones intro} turns out to be also a necessary condition for restricted diagonalization, when $N=\infty$.  In particular, 
Eq. (\ref{eq nuestras condiciones intro}) gives a complete characterization of those  diagonalizable operators which are $\cU_\cJ(\cH)$-diagonalizable for an arbitrary arithmetic mean closed operator ideal $\cJ$.  
We conclude the section with some  miscellaneous remarks and consequences.

In Section \ref{structure of dj} we obtain two results related with the way that the self-adjoint part of $\cD_\cJ$ sits inside the self-adjoint part of $\cJ$ (Theorems \ref{teo una pregunta1} and \ref{teo una pregunta2}).  By using previous results on restricted diagonalization, we further obtain a characterization of all the possible $\cU_\cJ(\cH)$-diagonalizations of arbitrary  $\cU_\cJ(\cH)$-diagonalizable operators (Theorem \ref{teo caract de las posibles diagonalizaciones}). 

\section{Preliminaries}\label{Prelim}

Let $\cH$ be a separable infinite-dimensional complex  Hilbert space and let $\cB(\cH)$ be the algebra of bounded operators acting on $\cH$. 
For an arbitrary operator $A \in \cB(\cH)$, we denote by $R(A)$ and $N(A)$ the range and nullspace of $A$, respectively.

\medskip

\noi \textbf{Operator ideals on Hilbert spaces.} Our main references here are the classic book by Gohberg and Krein \cite{GK60}, and the more recent approach in the work by Dykema, Figiel, Weiss, and Wodzicki \cite{Dyk04}. By an \textit{operator ideal} we mean a two-sided ideal $\cJ$ of $\cB(\cH)$. We say that the ideal is \textit{proper} when 
$\{0\}\neq \cJ\neq \cB(\cH)$. A result that goes back to Calkin  \cite{Cal41} states the inclusions  $\cF(\cH)\subseteq \cJ  \subseteq \cK(\cH)$ for a proper operator ideal $\cJ$, where $\cF(\cH)$ and $\cK(\cH)$ denote the ideals of finite-rank operators and compact operators on $\cH$, respectively.

Let $c_0\da=(c_0(\N)^+)\da$ denote the set of non-negative, non-increasing sequences in $c_0=c_0(\N)$. 
A \textit{characteristic set} $\Sigma$ is a positive cone of $c_0\da$, which  is hereditary  and invariant by ampliations. The latter means that given $\SE{\al}{n}\in \Sigma$, the sequence $\SE{(D_m\al)}{n}$ defined by
$$
(D_m\al)_n=\al_k  \, \, \, \text{ if } \, \, \, (k-1)m+1 \leq n \leq km, \, \, \, k \geq 1;
$$
also satisfies $\SE{(D_m\al)}{n} \in \Sigma$ for all $m \geq 1$.
It is well known  that there is a lattice isomorphism between  
characteristic sets and proper operator ideals (see, e.g. \cite[Thm. 4.2]{Dyk04}). Recall that the singular values of an operator $A\in \cK(\cH)$ are the eigenvalues of 
$|A|=(A^*A)^{1/2}$, 
counting multiplicities and arranged in a non-increasing sequence $s(A):=\{ s_n(A)\}_{n\geq 1}$.  Then one can assign to each characteristic set $\Sigma$ the proper operator ideal
$$
\cJ(\Sigma)=\{A\in \cK(\cH):  s(A) \in\Sigma\}.
$$ 
 Conversely, the inverse maps a proper operator ideal $\cJ$ to the characteristic set defined by 
$
\Sigma(\cJ)=\{s(A)\in c_0\da  :  A\in\cJ\}
$.

There are several properties of operator ideals that can be described in terms of its associated characteristic sets. An example of this situation is the notion of \textit{arithmetic mean closed operator ideal} that we now recall for later use. Given a sequence 
$\alpha=\SE{\al}{n}\in c_0\da$, the associated arithmetic mean sequence $\alpha_a=\SE{(\al_a)}{n}\in c_0\da$ is defined by 
$$
(\alpha_a)_n=\frac{1}{n}\ (\alpha_1+\ldots+\alpha_n), \, \, \,  n\geq 1\,.
$$ 
Given two sequences $\al=\SE{\al}{n} \in c_0\da$, $\beta=\SE{\beta}{n} \in c_0\da$, we write $\al=O(\beta)$ when $\al_n\leq M \beta_n$ for some constant $M>0$ and all $n\geq 1$.
Then for a proper operator ideal $\cJ\subset\cB(\cH)$, we consider the arithmetic mean closure of $\cJ$, which is denoted by $\cJ^{-\rm am}$, and is defined as the proper operator ideal given by
$$
\cJ^{-\rm am}=\{A\in\cK(\cH)\ : \ \exists \, \alpha\in\Sigma(\cJ)\ , \  s(A)_a= O(\alpha_a)\}\,.
$$
A proper operator ideal $\cJ$ is called \textit{arithmetic mean closed} if $\cJ^{-\rm am}=\cJ$. This notion becomes relevant due to the remarkable characterization of commutator spaces of operator ideals in terms of arithmetic means \cite{Dyk04}. We also refer to \cite{KW11} for a thorough investigation of arithmetic mean closed operator ideals. 

For instance, note that the ideal of finite-rank operators $\cF(\cH)$ is not 
arithmetic mean closed. Furthermore, it is not difficult to check that $\cF(\cH)^{- \rm am}=\fS_1(\cH)$, where $\fS_1(\cH)$ denotes the ideal of trace class operators. 

\begin{exas}\label{exam amc oideals}
 We list some  examples of arithmetic mean closed operator ideals.  
 Let $c_{00}=c_{00}(\N)$ be the real vector space consisting of all sequences with a finite number of nonzero terms.  A \textit{symmetric norming function} is a norm $\Phi: c_{00} \to \R$ satisfying the following properties: $\Phi(1,0,0,\ldots)=1$ and $\Phi(a_1, a_2 , \ldots , a_n, 0, 0 , \ldots)=\Phi(|a_{\sigma(1)}|, |a_{\sigma(2)}| , \ldots , |a_{\sigma(n)}|, 0, 0 , \ldots)$, where  $\sigma$  is any permutation of the integers $1,2,\ldots, n$ and $n\geq 1$.  Then, using the singular values of any  compact operator $A$ one can define: 
\[  \| A\|_{\Phi}:= \sup_{k \geq 1} \Phi (s_1(A), s_2(A), \ldots, s_k(A), 0, 0 , \ldots) \in [0,\infty].   \] 
 It turns out that 
\[  \fS_{\Phi}(\cH):= \{  \,   A \in \cK(\cH)    \, : \, \|A \|_{\Phi} < \infty   \, \} \]
is an arithmetic mean closed operator ideal as a consequence of the dominance property for singular values (see \cite[p.82]{GK60}). This includes the Schatten ideals $\fS_p(\cH)$ ($1\leq p \leq \infty$) if one takes $\Phi_p=\| \, \cdot \, \|_{\ell^p}$; in particular $\fS_\infty(\cH)=\cK(\cH)$ are the compact operators and $\Phi_\infty=\| \, \cdot \, \|_{\ell^\infty}$. 
Other examples  of arithmetic mean closed operator ideals are given by particular classes of Lorentz ideals, Marcinkiewicz ideals and Orlicz ideals;  see  \cite[Section 4.7]{Dyk04} for their definitions in terms of  characteristic sets and the precise statements about when they are arithmetic mean closed.
\end{exas}

Let $\cP=\{ P_n\}_{n=1}^N$ ($N\in\N$ or $N=\infty$) be a sequence of mutually orthogonal (self-adjoint) projections. The \textit{pinching operator induced by $\cP$} of an operator $A \in \cB(\cH)$ is given by
$$
\cC_\cP(A)=\sum_{n=1}^N P_n\,A\,P_n \, .
$$ 
In the case in which $N=\infty$, the above series is convergent in the strong operator topology (SOT). 

\begin{rem}\label{pinching properties}
Pinching operators have the following properties: 
\begin{itemize}
\item[i)]  If $A \in \cK(\cH)$, then   $\cC_\cP(A)\in\cK(\cH)$. Furthermore, in this case, the series defining $\cC_\cP(A)$ converges in the operator norm and
$$\sum_{n=1}^k s_n(\cC_\cP(A)) \leq \sum_{n=1}^k s_n(A),$$
for all $k \geq 1$ (see \cite[Thm. 5.2]{GK60}). In particular, $s(\cC_\cP(A))_a=O(s(A)_a)$. Therefore, if $\cJ$ is a proper operator ideal and $A\in\cJ$, then $\cC_\cP(A)\in\cJ^{-\rm am}$.
\item[ii)] Suppose now that the projections in $\cP=\{ P_n \}_{n \geq 1}$ satisfy  $\mathrm{rank}(P_n)=1$, for all $n \geq 1$, and $\sum_{n=1}^\infty P_n=I$. Then an operator ideal $\cJ$ is arithmetic mean closed if and only if $\cC_\cP(\cJ)\subset \cJ$ (\cite[Thm. 4.5]{KW11}).
\end{itemize}
\end{rem}

\medskip

\noi \textbf{Essential codimension.} Let $P,Q \in \cB(\cH)$ be two (orthogonal) projections  such that $P-Q\in \cK(\cH)$.  Under this assumption, the \textit{essential codimension} $[P:Q]$ of $P$ and $Q$ was introduced by Brown, Douglas and Fillmore \cite{BDF73} as the integer
given by
\begin{equation*}
[P:Q]  =  \begin{cases}
      \Tr(P)- \Tr(Q) &   \Tr(P)<\infty, \, \Tr(Q)<\infty;\\
      \mathrm{Ind}(V^*W) &    \Tr(P)= \Tr(Q)=\infty, WW^*=P,\, VV^*=Q,\, V^*V=W^*W=I.
					\end{cases}       
\end{equation*}
One can check that the operator $V^*W$ is Fredholm, and its index does not depend on the isometries $W,V$ such that $WW^*=P$ and $VV^*=Q$. Notice that the assumption $P-Q \in \cK(\cH)$ implies that the operator $QP|_{R(P)}:R(P)\to R(Q)$ is Fredholm. Furthermore, the essential codimension can be also computed as its Fredholm index 
\begin{align}
[P:Q]& =\mathrm{Ind}(QP|_{R(P)}:R(P)\to R(Q)) \label{index pair}\\
& = \dim(N(Q)\cap R(P)) - \dim(R(Q)\cap N(P)) \nonumber.
\end{align}
Avron, Seiler and Simon \cite{ASS94} defined the notion of \textit{index of a pair of projections} $(P,Q)$ as the previous index,  whenever the operator $QP|_{R(P)}:R(P)\to R(Q)$ is Fredholm.  In this case, the pair of projections $(P,Q)$ is said to be a \textit{Fredholm pair}.    
Simpler proofs of their results on the index of a pair of projections can be found in \cite{AS94}.  We observe that the notion of index of a Fredholm pair of projections  does not require  the difference of the projections being compact, which makes it  more general than the notion of essential codimension. 

From now on, we write $[P:Q]$ for the index in Eq. (\ref{index pair}) of a Fredholm pair of projections $(P,Q)$, and we shall call it the essential codimension of $P$ and $Q$, even when $P-Q \notin \cK(\cH)$.

\begin{rem}\label{prop ess cod}
We collect here some  facts on the essential codimension. The first two items follow easily, and the third one is proved in \cite[Thm. 3.4]{ASS94}.
\begin{itemize}
\item[i)] $[P:Q]=-[Q:P]$.
\item[ii)] Let $(P_i,Q_i)$, $i=1,2$, be two Fredholm pairs of projections.  If  $P_1P_2=0$ and $Q_1Q_2=0$, then $(P_1+ P_2,Q_1+Q_2)$ is a Fredholm pair and
$
[P_1 + P_2:Q_1 + Q_2]=[P_1:Q_1] + [P_2:Q_2].
$ 
\item[iii)] If $(P,Q)$ and $(P,R)$ are Fredholm pairs, 
and either $Q-R$ or $P-Q$ are compact, then $(P,R)$ is a Fredholm pair and 
$[P:R]=[P:Q] + [Q:R]$.
\end{itemize}
\end{rem}

\medskip

\noi \textbf{Restricted diagonalization.} In  this work, we fix $\e=\{ e_n\}_{n \geq 1}$ an orthonormal basis of $\cH$. An operator $A$ is  \textit{diagonal} (with respect to the fixed basis $\e$)  if   $\PI{Ae_n}{e_m}=\delta_{nm} \la_{n}$, for some bounded sequence of complex numbers $\{ \la_n\}_{n \geq 1}$. The algebra of diagonal operators is then given by 
$$\cD=\{ A\in\cB(\cH) :  A \text{ is a diagonal operator} \}\,.$$
 More generally, we say that an operator $A$ is \textit{diagonalizable}
if there exists a unitary $U$ such that $UAU^*\in\cD$. Notice that this notion does not depend on our choice of the fixed orthonormal basis $\e$.

Given $A \in \cB(\cH)$, we denote by 
$\sigma(A)$ (resp. $\sigma_p(A)$) the spectrum of $A$ (resp. point spectrum of $A$). 
Notice that a diagonalizable operator  must be normal. 
Conversely, 
a normal operator $A$ acting on the separable Hilbert space $\cH$, which has a spectral measure $E$, turns out to be diagonalizable if and only if $E(\sigma(A)\setminus \sigma_p(A))=0$. One can also check that $\sigma(A)=\overline{\sigma_p(A)}$, for every diagonalizable operator. 

Associated to each  diagonalizable operator $A$, there is a family  $\{\la_n \}_{n=1}^N$ $(N\in\N$ or $N= \infty)$, consisting of its distinct eigenvalues,   and the corresponding spectral projections  $\{ P_n\}_{n=1}^N$, which are defined by $P_n=E(\{ \la_n \})$. By applying the  spectral theorem for normal operators, 
we see that the
family $\{P_n\}_{n=1}^N$ is a decomposition of the identity, i.e. $P_n^*=P_n$, $P_n P_m=\delta_{nm} P_n$ and $\sum_{n=1}^N P_n=I$, where the series converges in the SOT when $N=\infty$. 
Also, the operator $A$ can be expressed as    
$$
A=\sum_{n=1}^N \la_n \, P_n \, ,
$$   
where  the above series is convergent in the SOT when $N=\infty$. 

\medskip

We denote by $\cU(\cH)$ the full unitary group of the Hilbert space $\cH$. Given an operator ideal 
$\cJ$, we consider the group defined by
$$
\cU_\cJ(\cH):=\{ U \in \cU(\cH): U- I \in \cJ \}.
$$
These types of groups are non trivial and strictly smaller subgroups  of the full unitary group if $\cJ$ is a proper operator ideal.
They can be studied  from a geometrical 
viewpoint for many operator ideals. For instance, if the operator ideal $\cJ$ admits a complete norm $\| \, \cdot \,\|_\cJ$ stronger than the operator 
norm, then $\cU_\cJ(\cH)$ turns out to be a real Banach-Lie group in the topology defined by $\| \, \cdot \, \|_\cJ$ (see \cite[Prop. 9.28]{B}).  For example, this condition is satisfied by the Schatten ideals $\fS_p(\cH)$ $(1\leq p \leq \infty)$. 

The following notion was introduced by Belti\c{t}$\breve{\text{a}}$, Patnaik and Weiss \cite{BPW16}:

\begin{fed}\label{defi restric diag diag op}
Let $\cJ$ be an operator ideal. An operator $A$ is called  \textit{$\cU_\cJ(\cH)$-diagonalizable (with respect to $\e$)}, or simply, \textit{restricted diagonalizable}, if there exists a unitary $U \in \cU_\cJ(\cH)$ such that $UAU^* \in \cD$.
\end{fed}

Let $\cI,\,\cJ$ be two operator ideals. It will be useful to consider as in \cite{BPW16} the  set of all $\cU_\cJ(\cH)$-diagonalizable operators belonging to $\cI$, i.e.
$$
\cD_{\cJ,\cI}:=\{ UDU^* : D \in \cD \cap \cI, \, U \in \cU_\cJ(\cH) \}=\bigcup_{U \in \cU_\cJ(\cH)} U(\cD \cap \cI)U^* \subseteq \cI.
$$ 
  In particular, we put 
$\cD_\cJ:=\cD_{\cJ,\cJ}$, and note that $\cD_{\cJ,\,\cB(\cH)}$ is the set of all $\cU_\cJ(\cH)$-diagonalizable operators.

The following result will play a key role in restricted diagonalization. It is actually a  reformulation  of a recent result by Kaftal and Loreaux \cite[Prop. 2.7]{KL17} (see also \cite[Prop. 2.3]{L19}).

\begin{pro}\label{cond proj res diag}
Let $P$, $Q$ be orthogonal projections  and let $\cJ$ be  a proper operator ideal. Then $P- Q \in \cJ$
 and $[P:Q]=0$ if and only if there is an unitary operator $U \in \cU_\cJ(\cH)$  such that $Q = UPU^*$.
\end{pro}

It is worth mentioning that this result was first obtained for the ideal $\cJ=\cK(\cH)$ in \cite{BDF73}, and also for the Hilbert-Schmidt operators $\cJ=\fS_2(\cH)$ in \cite{SV78}. Later on was generalized for the class of symmetrically-normed ideals in \cite{C85}; see \cite[p. 70]{GK60} for the definition of these ideals.  In particular, we observe that the ideals $\fS_\Phi(\cH)$ defined in Examples \ref{exam amc oideals} are symmetrically-normed ideals.

\begin{rem}
More recently, using the properties of the essential codimension stated in Remark \ref{prop ess cod}, Loreaux showed that Proposition \ref{cond proj res diag} can be reformulated as a result on restricted diagonalization of projections. He showed in  \cite[Corol. 3.1]{L19} that an (orthogonal) projection $P$ is $\cU_\cJ(\cH)$-diagonalizable if and only if there exists a diagonal projection $E\in\cD$ such that $P-E \in \cJ$. 
\end{rem}
	
The above remark was used by Loreaux to prove the following  characterization of restricted diagonalization of finite spectrum normal operators in terms of spectral projections (see \cite[Thm. 3.4]{L19}).

\begin{teo}\label{teo1 L19 original}
Let $\cJ$ be a proper operator ideal. A finite spectrum normal operator is $\cU_\cJ(\cH)$-diagonalizable if and only if each of its spectral projections
differs from a diagonal projection by an element of $\cJ$.
\end{teo}

The next result from \cite[Prop. 3.5]{BPW16}  will be useful in what follows.

\begin{pro}\label{pro 35 BPW16}
Let  $\mathrm{e}=\{e_n\}_{n\geq 1}$, $\mathrm{f}=\{f_n\}_{n\geq 1}$ be two orthonormal bases in $\cH$. Suppose that there exists $\delta>0$ such that $\|e_n-\gamma\,f_m\|\geq \delta$, for all $m,\,n\geq 1$ and $|\gamma|=1$. Then, for every operator $X$ which is diagonal with respect to $\mathrm e$ and with spectral multiplicities one, and for every $W\in\cU_{\cK(\cH)}(\cH)$, we have that $WXW^*$ is not a diagonal operator with respect to $\mathrm f$.
\end{pro}

\section{Restricted diagonalization}\label{restricted diagonalization}

We present  our  results on restricted diagonalization of operators with finite and infinite spectrum separately. First, we consider the case  of finite spectrum normal operators, which 
can be seen as a reformulation of Loreaux's result. Second, we consider in our main result the case of  (normal) diagonalizable operators with infinite point spectrum.

\begin{teo}[Finite  spectrum]\label{teo1 L19}
Let $\cJ$ be a proper operator ideal and let $A\in\cB(\cH)$ be a finite spectrum normal operator with spectral projections $\{P_n\}_{n=1}^N$. Then the following conditions are equivalent
\bit 
\item[i)] $A$ is $\cU_\cJ(\cH)$-diagonalizable; 
\item[ii)] There exist diagonal projections $E_n\in\cD$ such that 
$E_n-P_n\in\cJ$, $1\leq n\leq N$; 
\item[iii)] There exist diagonal projections $E_n\in\cD$, $1\leq n\leq N$, such that 
\begin{equation}\label{cond serie debil finitos}
 \sum_{n = 1}^N  (I-E_n) P_n \in \cJ \, \, \, \, \text{  and  } \, \, \, \,  \sum_{n =1}^N E_n(I- P_n ) \in \cJ. 
\end{equation}   
\eit
\end{teo}
\begin{proof}
The equivalence of items $i)$ and $ii)$ is due to Loreaux  as we have already stated in Theorem \ref{teo1 L19 original}.
 Assume that item $ii)$ holds, that 
is, there are diagonal projections $E_n\in\cD$ such that $P_n - E_n \in \cJ$, for all $n=1, \ldots, N$. 
Then, $(I-E_n)P_n=(P_n-E_n)P_n \in \cJ$ and $E_n(I-P_n)=E_n(E_n-P_n)\in \cJ$.
These facts clearly imply that both conditions in Eq. (\ref{cond serie debil finitos}) are satisfied. 
To prove the reverse direction, we use that
$\sum_{n=1}^N P_n=I$, and note that
$$
\sum_{n=1}^N E_n=\sum_{n=1}^N E_n(I-P_n) -\sum_{n=1}^N (I-E_n)P_n + I= I+ K,
$$
for some operator $K \in \cJ$. Since $\{ E_n \}_{n=1}^N$ are commuting projections, we have for $k \neq m$,
$$E_kE_m\,\sum_{n=1}^N E_n=
2E_kE_m + \sum_{n \neq k,\, m}E_k E_m E_n = E_k E_m + K' 
$$
for $K'=E_kE_mK \in \cJ$. This implies that $0\leq E_k E_m \leq K'$, and since $K'$ is compact, the projection $E_k E_m$ must have finite rank. Following the same argument given in \cite{L19} we construct a family of mutually orthogonal projections $\{ E_n'\}_{n=1}^N$
 as follows: set $E_1'=E_1$ and $E_n'=E_n- E_n(E_1' + \ldots E_{n-1}')$, where $n=1, \ldots , N$. 
Using that $E_k E_m$, $k\neq m$, has finite rank, we get that $E_n'$ and $E_n$ differ by a finite rank operator for each $n=1, \ldots , N$. 
Therefore the conditions in Eq. (\ref{cond serie debil finitos}) still hold if we replace the family $\{ E_n\}_{n=1}^N$ by $\{ E_n'\}_{n=1}^N$.
This allows us to multiply by $E_n'$ and $P_n$ in Eq. (\ref{cond serie debil finitos}) to find that $E_n'(I-P_n)\in \cJ$ and  $(I-E_n')P_n \in \cJ$ for all $n=1, \ldots,N$. Hence, $P_n - E_n'=P_n(I-E_n') - (I-P_n)E_n' \in \cJ$, which is equivalent to have $P_n - E_n \in \cJ$. 
\end{proof}

Actually, the above reformulation of Loreaux's condition in Eq. (\ref{cond serie debil finitos})  might be seen as a motivation to state the next case of diagonalizable operators with infinite point spectrum. We just let   $N \to \infty$ in Eq. (\ref{cond serie debil finitos}), and assume that the limits still belong to a proper 
operator ideal $\cJ$, where the limits can be taken in the 
weak operator topology. 
The resulting conditions are stated as 
follows.

\begin{teo}[Infinite point spectrum]\label{main result}
Let $\cJ$ be a proper operator ideal. Let $A\in\cB(\cH)$ be a  diagonalizable operator with  infinite point spectrum and spectral projections $\SE{P}{n}$. 
Assume that there is a family $\SE{E}{n}$ of diagonal projections such that 
\begin{equation}\label{eq cond caract1}
\sum_{n \geq 1}(I-E_n) P_n \in \cJ \, \, \, \, \text{  and  } \, \, \, \,  \sum_{n \geq  1}E_n (I-P_n)  \in \cJ, 
\end{equation}
where both series are assumed to converge in the WOT. 
Then $A$ is $\cU_\cJ(\cH)$-diagonalizable.

Conversely, if $A$ is $\cU_\cJ(\cH)$-diagonalizable and we assume further that $\cJ$ is arithmetic mean closed, then there is a family $\SE{E}{n}$ of diagonal projections such that the conditions in Eq. \eqref{eq cond caract1} hold.
\end{teo}

\begin{proof}
See Section \ref{sec prueba main result}.
\end{proof}

\begin{rem}\label{ejems AM ideales}
We point out that Theorem \ref{main result} 
provides with a complete characterization of those normal diagonalizable operators that are $\cU_\cJ(\cH)$-diagonalizable for a proper arithmetic mean closed operator ideal $\cJ$. In Examples \ref{exam amc oideals} we give a list of some of these ideals.
 Finally, we observe here that the assumption that $\cJ$ is arithmetic mean closed cannot be removed in the reverse implication (see Example \ref{counterexample amc}). 
\end{rem}

\begin{rem}\label{rem localizacion} 
Since compact normal operators are diagonalizable, the sufficient conditions in Eq. (\ref{eq cond caract1}) give an answer to \cite[Problem 6.4]{BPW16}, which asked for sufficient conditions for a normal operator in an proper operator ideal $\cJ$ to belong to $\cD_{\cJ}$.

On the other hand, let $\cI$ and $\cJ$ be two operator ideals and assume further that $\cJ$ is a proper ideal. Notice that 
Theorem \ref{main result} also allow us to derive results related with 
$$
\cD_{\cJ,\cI}:=\{ UDU^* : D \in \cD \cap \cI, \, U \in \cU_\cJ(\cH) \}
\subseteq \cI\,.
$$
For example, if $\cJ$ is a proper arithmetic mean closed ideal and $A\in\cB(\cH)$ is a diagonalizable operator with  infinite point spectrum and spectral projections $\SE{P}{n}$, then we have that $A\in \cD_{\cJ,\cI}$ if and only if $s(A)\in\Sigma(\cI)$ and there exists a family $\SE{E}{n}$ of diagonal projections such that the conditions in Eq. \eqref{eq cond caract1} hold.  
\end{rem}

\subsection{Proof of Theorem \ref{main result}}\label{sec prueba main result}

The proof of the forward direction of Theorem \ref{main result} is broken into the following lemmas and propositions.

\begin{lem}\label{proj diag}
 Let $\SE{E}{n}$ be a sequence of diagonal projections such that $\sum_{n\geq 1 } E_n=I+ K$, for some compact operator $K$.
Then there exists an integer $n_0 \geq 1$ satisfying the following  conditions:
\begin{itemize}
\item[i)] $E_n E_m =\delta_{nm}E_n$, for all $m,n > n_0$;
\item[ii)]  $E_n E_m$ is a finite-rank projection, for all $1\leq m, n\leq n_0$, $m\neq n$;
\item[iii)] For each $N \geq n_0$,  $E^{(N)}=\sum_{n > N}E_n$ is a diagonal projection, and $E^{(N)}E_n=0$, for all $n=1, \ldots, N$.
\end{itemize}
\end{lem}
\begin{proof}
\noi $i)$ The proof consists in a  slight modification of the argument given in the proof of Theorem 
\ref{teo1 L19}.  We include the details for completeness.   Consider the diagonal projections $Q_n:=\sum_{i=1}^n e_i \otimes e_i$, $n \geq 1$, where $e_i\otimes e_i$ denotes the rank-one projection onto $\mathrm{span}\{e_i\}$. Using that $K$ is a compact operator, it follows that there is an integer $n_1 \geq 1$ such that $\|(I-Q_{n_1})K(I-Q_{n_1})\|<1$.  
Next we observe that there exists an integer $n_0 \geq 1$ such that $R(E_n) \subseteq N(Q_{n_1})$, for all $n \geq n_0$. 
In fact, this follows immediately using $\|\sum_{n \geq 1}E_n\| < \infty$.

We now prove that $E_n E_m =\delta_{nm}E_n$, for all $m,n > n_0$. Notice that this is equivalent to $R(E_n)\cap R(E_m) = \{ 0\}$,  for all $m,n > n_0$, $n \neq m$. We proceed by way of contradiction and assume that there is a vector $x \in R(E_n)\cap R(E_m)$, $\| x \|=1$. Consider the operator $C=(I-Q_{n_1})(\sum_{k \neq n,m}E_k)(I-Q_{n_1}) \geq 0$. From the assumption $\sum_{n \geq 1 }E_n=I+ K$, and since $R(E_n) \subseteq N(Q_{n_1})$ and $R(E_m) \subseteq N(Q_{n_1})$, we get that
$$
(E_n + E_m) + C=(I-Q_{n_1}) + (I-Q_{n_1})K(I-Q_{n_1}).
$$
 Thus, we have a contradiction
\begin{align*}
2 & = \PI{ (E_n +E_m)x}{x} \leq \PI{(E_n +E_m + C)x}{x} \\
& =\PI{(I-Q_{n_1} + (I-Q_{n_1})K(I-Q_{n_1})) x}{x}<2,
\end{align*}
and this finishes the proof.

\medskip

\noi $ii)$ The same argument given in the proof of Theorem \ref{teo1 L19} can be repeated in this context.

\medskip

\noi $iii)$ Similar arguments to that of the first item also work in this item. 
 \end{proof}

\begin{lem}\label{finite diff}
Let $\cJ$ be a proper operator ideal and let $\SE{P}{n}$ be a decomposition of the identity. Suppose that there is a family  of diagonal projections  $\SE{E}{n}$ such that 
\begin{equation}\label{cond serie debil}
 \sum_{n \geq 1}(I-E_n) P_n \in \cJ \, \, \, \, \text{  and  } \, \, \, \,  \sum_{n \geq  1}E_n(I- P_n ) \in \cJ, 
\end{equation}
where both series are assumed to converge in the WOT. Then  $P_n - E_n \in \cJ$, for all $n \geq 1$.
\end{lem}
\begin{proof}
Set $K=\sum_{n\geq 1}  (I-E_n) P_n \in \cJ$ and $L=\sum_{n\geq 1}E_n(I- P_n ) \in \cJ$. Since $\sum_{n\geq 1} P_n=I$, the series $S:=\sum_{n= 1}^\infty E_nP_n$ converges weakly, and $S=I-K$.
Therefore,
$$
\sum_{n\geq 1} E_n=L + S=I + L- K,
$$
where $L-K\in\cJ$ is a compact operator. Applying Lemma \ref{proj diag}, we find an integer $n_0 \geq 1$ such that $E_n E_m =\delta_{nm}E_n$, for all $m,n > n_0$, and $E^{(N)}=\sum_{n > N}E_n$ is a diagonal projection such that $E^{(N)}E_n=0$, for all $n=1, \ldots, N$, $N\geq n_0$. We can multiply the series in Eq. (\ref{cond serie debil}) by the projections   $P^{(N)}=\sum_{n > N} P_n$ and $E^{(N)}$  to obtain 
\begin{equation}\label{cond serie debil2}
 \sum_{n =1}^{N}  (I-E_n) P_n \in \cJ \, \, \, \, \text{  and  } \, \, \, \,  \sum_{n = 1}^{N} E_n(I- P_n ) \in \cJ. 
\end{equation}
We recall the argument from \cite{L19} to construct mutually orthogonal diagonal projections: let $E_1'=E_1$ and inductively define $E_n'=E_n - E_n(E_1' + \ldots + E_{n-1}')$, for all $n=2, \ldots , N$. 
Then, $E_1', \ldots , E_N'$ are  mutually orthogonal  diagonal projections. Again by Lemma \ref{proj diag} (items $i)$ and $ii)$) $E_n E_m$ is a finite-rank projection, for   $n \neq m$; hence, we see that $E_n'=E_n + R_n$, where $R_n$ is a finite-rank operator, $n=1, \ldots, N$. From this fact and Eq. (\ref{cond serie debil2}), it follows that 
\begin{equation*}
 \sum_{n =1}^{N}  (I-E_n') P_n \in \cJ \, \, \, \, \text{  and  } \, \, \, \,  \sum_{n = 1}^{N} E_n'(I- P_n ) \in \cJ. 
\end{equation*}  
Now we can multiply by $P_n$ and $E_n'$ to get that $E_n'(I-P_n) \in \cJ$ and $(I-E_n')P_n \in \cJ$ for $n=1, \ldots, N$. Hence, 
$$
P_n - E_n'=P_n(I-E_n') - (I-P_n)E_n' \in \cJ.
$$
Recall that $E_n'$ and $E_n$ differ by a finite-rank operator. Thus, we obtain $P_n - E_n \in \cJ$, for $n=1, \ldots, N$, where $N\geq n_0$ is arbitrary. The proof is completed. 
\end{proof}

\begin{rem}\label{several rems}
Consider the notation in Lemma \ref{finite diff}.
Then, the diagonal projections $\{ E_n \}_{n \geq 1}$ satisfying the conditions in Eq. (\ref{cond serie debil}) can be replaced by other diagonal projections $\{E'_n\}_{n\geq 1}$ which also satisfy the conditions in Eq. (\ref{cond serie debil}) and are mutually orthogonal. To this end, take  $E_1', \ldots, E_N'$  the mutually orthogonal diagonal projections defined in the proof of Lemma \ref{finite diff},  where $N$ is fixed, $N\geq n_0$. According to the property stated in Lemma \ref{proj diag} $iii)$ we can set $E_n'=E_n$ for $n>N$. Taking into account that $E_n$ and $E_n'$ differ by a finite-rank operator for $n=1, \ldots, N$, we get that $\{ E_n'\}_{n \geq 1}$ is a sequence of mutually orthogonal projections such that
\begin{equation*}
 \sum_{n \geq  1}(I-E_n') P_n \in \cJ \, \, \, \, \text{  and  } \, \, \, \,  \sum_{n \geq 1}E_n'(I- P_n ) \in \cJ, 
\end{equation*}
where both series converge weakly. 
\end{rem}

Our next result is a generalization of \cite[Thm. 2.3]{EC10} to an arbitrary proper operator ideal. This result was proved in the aforementioned work for separable symmetrically-normed ideals, 
 based on work by Carey \cite{C85} and by Serban and Turcu \cite{ST07}. Our approach is an adaption of the arguments of \cite{EC10}; we will need the following result which is \cite[Thm. 2.2]{ST07}.

\begin{lem}\label{lem ST}
Let $\cH$ and $\cK$ be separable, infinite dimensional Hilbert spaces. Let $\cH_1$ and $\cH_2$ be infinite dimensional subspaces of $\cH$ and let $P_i$ denote the orthogonal projection onto $\cH_i$, $i=1,2$. The following statements are equivalent:
\ben
\item There exist isometries $V_1,\,V_2\in \cB(\cK,\cH)$ with ranges $\cH_1$ and $\cH_2$ such that $V_1-V_2$ is compact.
\item $P_1-P_2$ is compact and $[P_1:P_2]=0$. \qed
\een 
\end{lem}

\begin{pro}\label{partial iso charac}
Let $V_1, V_2$ be partial isometries and let $\cJ$ be a proper operator ideal.  Then 
there exists $U \in \cU_\cJ(\cH)$ such that $UV_1=V_2$ if and only if $V_1-V_2 \in \cJ$ and 
$N(V_1)=N(V_2)$.
\end{pro}

\begin{proof}
If there exists $U \in \cU_\cJ(\cH)$ such that $UV_1=V_2$, then it is clear that $N(V_1)=N(V_2)$, and  
$V_1-V_2=V_1-UV_1= (I-U)V_1\in\cJ$.

Conversely, assume that $V_1, V_2$ are partial isometries such that $V_1-V_2 \in \cJ$ and 
$N(V_1)=N(V_2)$. We consider the following two cases. Assume first that $	\dim R(V_1)=\infty$. 
Since $R(V_1^*)=N(V_1)^\perp=N(V_2)^\perp$, then $U_1=V_2V_1^*:R(V_1)\rightarrow R(V_2)$ is a surjective isometry.
Hence, $\dim R(V_2)=\infty$ and $V_1|_{N(V_1)^\perp},\,V_2|_{N(V_1)^\perp}\in \cB(N(V_1)^\perp,\cH)$ are isometries such that $V_1|_{N(V_1)^\perp}-V_2|_{N(V_1)^\perp}$ is compact.
Then, by Lemma \ref{lem ST} we see that $P_1-P_2$ is compact and also $[P_1:P_2]=0$, where $P_i=V_iV_i^*$  is the orthogonal projection onto $R(V_i)=R(V_i|_{N(V_1)^\perp})$, for $i=1,2$. Moreover, since $V_1-V_2\in\cJ$, then 
$P_1-P_2=V_1V_1^*-V_2V_2^*= V_1(V_1^*-V_2^*)+(V_1-V_2)V_2^*\in\cJ$. 
 By Proposition \ref{cond proj res diag} we conclude that there exists $U_2\in\cU_\cJ(\cH)$ such that $U_2 P_1 U_2^*=P_2$, so then $U_2 (I-P_1) U_2^*=I-P_2$. In this case, the restriction $U_2|_{R(V_1)^\perp}:R(V_1)^\perp\rightarrow R(V_2)^\perp$ is a surjective isometry. We now set $U=U_1\oplus U_2|_{R(V_1)^\perp}\in\cU(\cH)$. Then, by construction $UV_1=V_2$. On the other hand, $(U-I)P_1=(V_2-V_1)V_1^*\in\cJ$ and $(U-I)(I-P_1)=(U_2-I) (I-P_1)\in\cJ$ which show that $U-I\in\cJ$ so $U\in\cU_\cJ(\cH)$.

Assume now that $\dim R(V_1)=k<\infty$. As before, we can define the surjective isometry $U_1=V_2V_1^*:R(V_1)\rightarrow R(V_2)$ so that $\dim R(V_2)=k<\infty$. Hence, $\dim (R(V_1)^\perp\cap R(V_2)^\perp )=\infty$ and we can choose an orthonormal basis $\{f_n\}_{n\geq 1}$ of $R(V_1)^\perp\cap R(V_2)^\perp$. Let 
$\{g_n\}_{n\geq 1}$ be an orthonormal basis of $\cH$ such that the first $k$ vectors form an  orthonormal basis of $N (V_1)^\perp=N (V_2)^\perp$. Set
$$\tilde V_1 \,g_n= \begin{cases} V_1\,g_n  \, ,&  1\leq n\leq k\,, \\ f_{n-k} \,,&\ \  n>k \,,\end{cases}
\peso {and }
\tilde V_2 \,g_n= \begin{cases} V_2\,g_n\,, &  1\leq n\leq k \,,\\ f_{n-k}\,, &\ \ n>k \,.\end{cases}$$
By construction, $\tilde V_1,\,\tilde V_2$ are isometries such that $\dim R(\tilde V_1)=\infty$ and $\tilde V_1-\tilde V_2\in\cJ$, since it is a finite rank operator. Then, by the previous case, we conclude that there exists $U\in\cU_\cJ(\cH)$ such that $U\tilde V_1=\tilde V_2$; this last fact clearly implies that 
$U V_1=V_2$.
\end{proof}

We observe that Proposition \ref{partial iso charac} may be considered as a version of Proposition \ref{cond proj res diag} for partial isometries instead of projections. Here the action of the group $\cU_\cJ(\cH)$ on the set of partial isometries moves the final space of the partial isometries (see \cite[Thm. 2.4]{EC10} for   an analogous result about the action moving  both the initial and final space). As we shall see, the characterization in Proposition \ref{partial iso charac} is key to our next result.

\begin{pro}\label{partial isom}
Under the same assumptions of Lemma \ref{finite diff},  then there exist an integer $n_1 \geq 1$ and a partial isometry $V\in \cB(\cH)$ with initial projection $P:=\sum_{n > n_1} P_n$ and final projection $E:=\sum_{n > n_1} E_n$ such that $$VP_n V^*=E_n \peso{for all} n > n_1\ ,$$  $V-P\in \cJ$, $E-P\in \cJ$  and $[E : P]=0$.
\end{pro}
\begin{proof}
We consider the operators $K ,L \in \cJ$  and the integer $n_0 \geq 1$ defined in the proof of Lemma \ref{finite diff}, that is, $n_0$ satisfies the conditions $i),\,ii)$ and $iii)$ from Lemma \ref{proj diag},
$$K=\sum_{n\geq 1}  (I-E_n) P_n \in \cJ \py L=\sum_{n\geq 1}E_n(I- P_n ) \in \cJ\,.$$
We claim that there is an integer $n_1 \geq n_0$ satisfying the following conditions:
\begin{enumerate}
\item[1.] $\|\sum_{n > n_1} P_n (I-E_n)P_n \|<1$; 
\item[2.] $\|\sum_{n > n_1} E_n (I-P_n)E_n \|<1$.
\end{enumerate}
To see this, note that the first item follows by using that the series $\sum_{n =1}^\infty P_n (I-E_n)P_n=\sum_{n = 1}^\infty P_n KP_n$ is convergent in the operator norm (see Remark \ref{pinching properties}). For the second item, note that $\{ E_n\}_{n >n_0}$ is a sequence of mutually orthogonal projections and $\sum_{n = 1}^{n_0} E_n(I-P_n) \in \cJ$ by Eq. (\ref{cond serie debil2}) with $N=n_0$. Therefore, $L'=\sum_{n >n_0} E_n(I-P_n) \in \cJ$. Thus we can apply the same argument of the first item because the series $\sum_{n > n_0} E_n (I-P_n)E_n=\sum_{n >n_0}E_n L'E_n$  is also norm convergent. This proves our claim.

Now we set  
$$
 P =\sum_{n > n_1} P_n \, ; \, \, \, \, E=\sum_{n > n_1} E_n \, ; \, \, \, \, S=\sum_{n > n_1} E_nP_n\, ;
$$
where all the series  converge strongly. Indeed, in order to see that the third series converges strongly notice that if $m>n_1$ then
$$
\left\| \sum_{n\geq m} E_nP_nx \right\|^2=\sum_{n\geq m}\|E_nP_nx\|^2\leq \sum_{n\geq m}\|P_nx\|^2\xrightarrow[m\rightarrow\infty]{}0\,,
$$where we have used that $\{E_n\}_{n>n_1}$ and $\{P_n\}_{n>n_1}$ are families of mutually orthogonal projections.  Observe that, by construction, $ESP=S$. 
On the other hand, using that convergence in the SOT implies convergence in the WOT and that the involution is WOT-continuous we see that $S^*=\sum_{n>n_1}P_nE_n$, where the convergence is in the WOT. Arguing as before, we see that the last series actually converges in the SOT. Now notice that the sequence of partial sums 
$S_k:=\sum_{n = n_1 + 1}^k E_nP_n$, $k \geq n_1 + 1$, is norm bounded. Indeed, for $x \in \cH$, we have
$$\|S_k x\|^2=\sum_{n = n_1 + 1}^k \|E_n P_n x\|^2 \leq \sum_{n = n_1 + 1}^k \|P_n x \|^2 =\|x\|^2\,.$$ Thus, $\|S_k\| \leq 1$, $k \geq n_1 + 1$.   
Since multiplication is jointly SOT continuous on bounded sets of $\cB(\cH)$ we find that  $S^*S=\sum_{n > n_1}P_nE_nP_n$ and  $SS^*=\sum_{n > n_1}E_n P_n E_n$ converge in the SOT. 

 Notice that by the first inequality in item 1 above,
$$
\left\| |S|^2 - P \right\|= \left\|\sum_{n > n_1} P_n(I-E_n)P_n \right\|<1,
$$
which gives that $|S||_{R(P)}:R(P)\to R(P)$ is an invertible operator. A similar computation with $|S^*|^2$ and $E$ in place of $|S|^2$ and $P$ (based on item 2. above)
allows us to conclude that $|S^*||_{R(E)}:R(E)\to R(E)$ is invertible. Therefore $S|_{R(P)}:R(P) \to R(E)$ is also invertible. 
Notice that $V:=S|S|^\dagger$ is a partial isometry with initial projection $P$ and final projection $E$,  where 
 $|S|^\dagger$ denotes the Moore-Penrose inverse of $|S|$. Now note that
$$
S-P=-\sum_{n> n_1}(I-E_n)P_n=-KP \in \cJ\,.
$$ 
Since $S-P\in\cJ$, then a direct computation shows that $|S|^2-P=S^*S-P^*P:=K'\in\cJ$.
Then,  we get $|S|- P=(|S|+P)^\dagger K' \in \cJ$, and consequently, $|S|^\dagger-P \in \cJ$.  This leads us to conclude $V-P=(S-P)|S|^\dagger + |S|^\dagger - P \in \cJ$.

We now show that $VP_nV^*=E_n$ for all $n > n_1$. Hence, we fix $n>n_1$ and note that $SP_n=E_nP_n=E_nS$, which gives $P_n|S|^2=|S|^2P_n$. This implies
that $P_n|S|=|S|P_n$ and $P_n|S|^\dagger=|S|^\dagger P_n$. Hence 
$$
VP_nV^*=S|S|^\dagger P_n |S|^\dagger S^*=SP_n (|S|^\dagger)^2S^*=E_n S (|S|^\dagger)^2S^*=E_nE=E_n\,.
$$
Finally, we note that $V$ and $P$ are two partial isometries such that
$N(V)=N(P)$ and $V-P\in \cJ$. Therefore we can apply Proposition \ref{partial iso charac} to conclude that there exists $U \in \cU_\cJ(\cH)$ such that $UV=P$. Thus, $UEU^*=UVV^*U^*=P$, which shows that $P-E \in \cJ$ and $[P:E]=0$  by Proposition \ref{cond proj res diag}.
\end{proof}

Now we can prove our main result.

\begin{proof}[Proof of Theorem \ref{main result}] 
We begin by showing the forward direction. Notice that we can apply Proposition \ref{partial isom} and obtain $n_1 \geq 1$, a partial isometry $V\in \cB(\cH)$ with initial projection
$P:=\sum_{n > n_1} P_n$ and final projection $E:=\sum_{n > n_1} E_n$ such that $VP_n V^*=E_n$, for all $n > n_1$, and such that $V-P\in \cJ$,   $P-E\in\cJ$ and $[P:E]=0$. According to Lemma \ref{finite diff}, we have $P_n - E_n \in \cJ$ for $n=1,\ldots, n_1$. Furthermore, as we have observed in Remark \ref{several rems}, by replacing the initial projections $E_1,\ldots,E_{n_1}$ (since $n_0\leq n_1$) we can assume that the diagonal projections $\{ E_n\}_{n \geq 1}$ are mutually orthogonal.
We further replace
$E_{n_1}$ by $$E_{n_1}:=I-\sum_{n\neq n_1}E_n=I-\left(\sum_{n=1}^{n_1-1}E_n+E\right)\,.$$ 
In this case, using that $\{P_n\}_{n\geq 1}$ is a decomposition of the identity,
$$
P_{n_1}-E_{n_1}=I-\left(\sum_{n=1}^{n_1-1}P_n+P\right) - \left( I-\left(\sum_{n=1}^{n_1-1}E_n+E\right)\right)=\sum_{n=1}^{n_1-1}E_n-P_n+E-P\in\cJ\,.
 $$ In this way, we can further assume that $\{E_n\}_{n\geq 1}$ is a decomposition of the identity in $\cD$.
Now, we set $E_0=E$, $P_0=P$,  and consider the finite families of projections
$$
\cE=\{E_0,\,E_1,\ldots,E_{n_1}\} \py \cP=\{P_0,\,P_1,\ldots,\,P_{n_1}\}\,.
$$
Notice that $\cP$ and $\cE$ are (finite) decompositions of the identity, and then by Remark \ref{prop ess cod}, it follows that
$$
\sum_{n=0}^{n_1}[P_n:E_n]=[I:I]=0\,.
$$
 Hence,
by applying \cite[Lemma 3.3]{L19} to $\cP$ and $\cE$ we obtain a family
$\cE'=\{E'_n\}_{n=0}^{n_1}$ such that $E'_n-P_n\in\cJ$ and 
$[P_n:E'_n]=0$, for $0\leq n\leq n_1$. By inspection of the proof of that result,
we further see that $E'_0=E$, since $[P_0:E_0]=[P:E]=0$.
 By \cite[Lemma 3.2]{L19}, we get that there exists
$U_0\in\cU_\cJ(\cH)$ such that $U_0P_nU_0^*=E'_n$, for $0\leq n\leq n_1$. 
Consider the partial isometry $V_\perp=  U_0 (I-P_0)$. Since $U_0\in \cU_\cJ(\cH)$, we
see that $V_\perp-(I-P_0)\in\cJ$, and by 
construction, $V_\perp P_n\,V_\perp^*=E'_n$ for $1\leq n\leq n_1$.

We set $U=V+V_\perp\in\cU(\cH)$. Notice that
$$
U-I=(V-P_0)+(V_\perp-(I-P_0))\in\cJ,
$$ 
$UP_n\,U^*=E'_n\in\cD$ for $1\leq n\leq n_1$, and $UP_n\,U^*=E_n\in\cD$, for $ n > n_1$. Hence, 
$$
UAU^*=\sum_{n\geq 1} \la_n\ UP_n\,U^*\in\cD,
$$
that is, $A$ is $\cU_\cJ(\cH)$-diagonalizable.

\medskip

We now show the reverse implication. Thus, we assume further that $\cJ$ is an arithmetic mean closed operator ideal and that there exists $U\in\cU_\cJ(\cH)$ such that 
$B=UAU^*\in\cD$. Therefore, $E_n:=UP_nU^*\in\cD$, for $n \geq 1$, since these are spectral projections of $B\in\cD$. Notice that in this case we get that $\{E_n\}_{n \geq 1}$ is  a decomposition of the identity. Moreover, we have that
\begin{align}
\sum_{n\geq 1} (I-E_n) P_n  & = \sum_{n\geq 1} (I-E_n) (U^* -I)  E_n U    
\nonumber
\\
& = (U^*-I)  \sum_{n\geq 1} E_n U -  \left(\sum_{n\geq 1} E_n \,(U^*-I)\,E_n \right)\, U  \nonumber \\
& = (U^*-I) U -  \left(\sum_{n\geq 1} E_n (U^*-I)E_n \right)\,U\in \cJ \label{series varias},
\end{align}
where we use that $U^*-I \in \cJ$ and that, in the second term, pinching operators  preserve arithmetic mean closed ideals (see Remark \ref{pinching properties}). Notice that the convergence of the first term is indeed in the operator norm. This is a consequence of the well-known fact that multiplication by  a compact operator of a strong convergent sequence turns  the resulting sequence into a norm convergent one. Also note that the second term   corresponding to the pinching of $U^*-I\in\cJ$ actually converges in the operator norm again by Remark \ref{pinching properties}. Hence, we conclude that the first series in Eq. (\ref{eq cond caract1}) is norm convergent to an operator belonging to $\cJ$. The second series in Eq. (\ref{eq cond caract1}) can be treated analogously. 
\end{proof}

\subsection{Remarks and consequences}

Some remarks on  Theorem \ref{main result} are in order.

\begin{rem}\label{additional rem}
$i)$ For a diagonalizable operator $A$ with infinite point spectrum and spectral projections $\SE{P}{n}$, one might ask whether the existence  of diagonal projections $\SE{E}{n}$ satisfying $P_n - E_n \in \cJ$ for all $n\geq 1$, implies that $A$ is $\cU_\cJ(\cH)$-diagonalizable. That is,  a generalization for the case of infinite point spectrum of the second condition in Theorem \ref{teo1 L19}.
The following simple example shows that the answer is in the negative.
Take $\mathrm{f}=\{ f_n\}_{n \geq 1}$ an orthonormal basis. Consider the rank-one projections defined by 
$P_n=f_n\otimes f_n$ and $E_n=e_n \otimes e_n$. Thus, $P_n- E_n \in \cF(\cH)\subseteq \cJ$, $n \geq 1$. Suppose that 
$A$ is any diagonalizable operator with simple spectrum (i.e. spectral multiplicities one) and  spectral projections $\SE{P}{n}$.
Assume further that 
the orthonormal basis $\mathrm{f}$ satisfies that
$\|f_n - \alpha e_n\| \geq \delta$, for some $\delta >0$ and all $\alpha \in \mathbb{T}$, $n \geq 1$ (we can always construct such $\mathrm{f}$, see \cite[Prop. 3.6]{BPW16}). Then, 
according to Proposition \ref{pro 35 BPW16}, there is no unitary $U\in \cU_{\cK(\cH)}(\cH)$ such that $UAU^*\in \cD$. If $\cJ$ is a proper operator ideal then $\cJ \subseteq \cK(\cH)$ and hence $A$ cannot be $\cU_\cJ(\cH)$-diagonalizable.

\medskip

\noi $ii)$ On the other hand, we remark that the condition $P_n-E_n \in \cJ$ for all $n\geq 1$ was derived from the 
conditions in Eq. \eqref{eq cond caract1} in Lemma \ref{finite diff}. Hence by Theorem \ref{main result} one can also deduce the existence of diagonal projections $\SE{E}{n}$ such that $P_n-E_n \in \cJ$ for all $n\geq 1$, whenever $A$ is $\cU_\cJ(\cH)$-diagonalizable with spectral projections $\SE{P}{n}$ and $\cJ$ is a proper arithmetic mean closed operator ideal.

\medskip

\noi $iii)$ The series in 
Eq. (\ref{eq cond caract1}) are supposed to converge weakly. Of course, this is weaker than convergence in the operator norm, and thus leads to an easier criterion to verify. However, an inspection of the proof of Theorem \ref{main result}, in particular the convergence in (\ref{series varias}), reveals that if both series satisfy (\ref{eq cond caract1}) with weak convergence, then it turns out that they must converge in the operator norm. 

\medskip

\noi $iv)$ Let $\cJ$ be an (arbitrary) proper operator ideal, let $A\in\cB(\cH)$ with infinite point spectrum and assume that there exists $U\in \cU_\cJ(\cH)$ such that $UAU^*\in\cD$, i.e. $A$ is $\cU_\cJ(\cH)$-diagonalizable. Let $\{P_n\}_{n\geq 1}$ denote the spectral projections of $A$ and let $E_n=UP_nU^*\in\cD$, for $n\geq 1$. Then, an inspection of Eq. \eqref{series varias} together with Remark \ref{pinching properties} show that, in general, $\sum_{n\geq 1} (I-E_n) P_n  \in\cJ^{-\rm  am}\subseteq \cK(\cH)$, where $\cJ^{-\rm  am}$ denotes the arithmetic mean closure of $\cJ$ (see Section \ref{Prelim}); in particular, we get  that the previous series determines a compact operator. A similarly argument shows that   
$\sum_{n\geq 1} E_n (I- P_n)  \in\cJ^{-\rm  am}$. 
\end{rem}

Next, we  show that the assumption that $\cJ$ is an arithmetic mean closed ideal cannot be removed  
 in the second assertion of Theorem \ref{main result}.

\begin{exa}\label{counterexample amc}
Let $\cF:=\cF(\cH)$ be the ideal of finite rank operators, whose arithmetic mean closure is $\cF^{-\mathrm{am}}=\fS_1(\cH)$, the ideal of trace class operators.  Recall that $\e=\{ e_n\}_{n \geq 1}$ is our fixed orthonormal basis. Take  $X=f \otimes f$, where $\|f\|=1$ and $\PI{f}{e_n}\neq 0$, for $n\geq 1$. Now consider $U=e^{iX}=I + (e^i-1)X \in \cU_{\cF}(\cH)$ and let $\{ P_n\}_{n \geq 1}$ be the  sequence of projections defined by $P_n=U(e_n\otimes e_n) U^*$, for $n \geq 1$. Let $\SE{\la}{n}$ be a bounded sequence of complex numbers such that $\la_n\neq \la_m$, for $n\neq m$, and let
$$
A=\sum_{n\geq  1}\la_n\, P_n\in\cD_{\cF, \,\cB(\cH)}\,.
$$ 
Assume that there exists a sequence $\{E_n\}_{n\geq 1}$ of diagonal projections such that 
\beq\label{eq cond finit1}
\sum_{n\geq  1}(I-E_n)P_n \in \cF \, \, \, \, \, \text{ and } \, \, \, \, \, \sum_{n\geq  1}E_n (I-P_n) \in \cF\,,
\eeq where we assume
that the convergence is in the operator norm 
(see Remark \ref{additional rem} $iii)$).  
In particular, 
we see that
$$
\|P_n-E_n\|\leq \|(I-E_n)P_n\|+\|E_n(I-P_n)\|\xrightarrow[n\rightarrow \infty]{} 0\,.
$$
Let $n_0 \geq 1$ be such that $\|P_n-E_n\|<1/2$ and $|\langle f,e_n\rangle |<1/2$ for $n\geq n_0$. Then, we have that $E_n=e_n\otimes e_n$, for $n\geq n_0$. Indeed, assume that $e_n\notin R(E_n)$. In this case $E_ne_n=0$, and since $Ue_n\in R(P_n)$ is a unit vector, we get that 
\begin{align*}
1/2 & >\|P_n-E_n\|\geq \|Ue_n-E_n Ue_n\| \\ & 
=\|e_n+(e^i-1)\langle e_n,f\rangle (f-E_nf)\|  \geq 1- |\langle e_n,f\rangle|>1/2, 
\end{align*}
where we have used that $|(e^i-1)|\,\|(f-E_nf)\|\leq 1$. The previous contradiction proves that $e_n\in R(E_n)$. On the other hand, if $e_m\in R(E_n)$ for some $m\neq n$, then
\begin{align*}
1/2 & >\|P_n-E_n\|\geq \|P_n e_m-e_m\| \\
& =\|(e^i-1)\langle e_n,f\rangle \langle e_m,f\rangle \, Ue_n - e_m\|
\geq 1- |\langle e_n,f\rangle|>1/2
\end{align*}
where we have used that $\|(e^i-1)\langle e_m,f\rangle\ Ue_n\|\leq 1$.

Recall that by Lemma
\ref{finite diff} we know that $E_n-P_n\in\cF$ for $n\geq 1$.
Since $P_n$ is a rank-one projection and $E_n-P_n$ is a finite rank operator, we see that $E_n$ is a finite rank projection. These facts show that we can replace the initial projections $E_n$ by $e_n\otimes e_n$ for $1\leq n\leq n_0-1$ and the resulting sequence (that we still denote by) $\{E_n=e_n\otimes e_n\}_{n\geq 1}$ is a decomposition of the identity that also verifies the conditions in Eq. \eqref{eq cond finit1}.

Notice that the pinching $\cC_\cE(X)$ 
of $X$ with respect to $\cE=\{E_n\}_{n\geq 1}$
is given by
$$
\cC_\cE(X)=\sum_{n \geq 1}E_nXE_n =\sum_{n \geq 1}|\PI{f}{e_n}|^2 E_n.
$$
Observe that $\cC_\cE(X) \in \fS_1(\cH)\setminus \cF$,
so then $\cC_\cE(U-I) \notin \cF$. 
Next we note that 
$$
E_n(I-P_n)=E_n(E_n-P_n)=-E_n((U-I)E_nU^* + E_n(U^*-I)),
$$
which implies 
$$
\sum_{n \geq 1}E_n(I-P_n)=-\cC_\cE(U-I)U^* + U^*-I = -(\cC_\cE(U-I)+ (e^{i}-1)X) U^*\notin \cF.
$$
This last fact contradicts Eq. \eqref{eq cond finit1}. We can proceed similarly with the other series. 
\end{exa}

\medskip

We end this section with some consequences of Theorem \ref{main result}: we obtain 
a characterization of operators $A\in\cB(\cH)$ that are $\cU_\cJ(\cH)$-diagonalizable for proper arithmetic mean closed operator ideals $\cJ$ with some further properties and a characterization of the existence of $U\in\cU_\cJ(\cH)$ that conjugate two given decompositions of the identity. We remark that in case $A$ is a finite spectrum operator or in case the decompositions of the identity are finite, then these results follow directly from Theorem \ref{teo1 L19 original}. Hence, we consider their respective infinite cases.

Recall that given an operator ideal $\cJ$, then $\cJ^2:=$ Span$\{ AB: A, B \in \cJ \}$ is also an operator ideal.

\begin{cor}\label{coro main result}
Let $\cJ$ be a proper operator ideal such that $\cJ^2$ is arithmetic mean closed.  Let $A$ be a  diagonalizable operator with infinite spectrum and with  spectral projections $\{ P_n\}_{n\geq 1}$. The following conditions are equivalent:
\begin{enumerate} 
\item[i)] $A$ is $\cU_\cJ(\cH)$-diagonalizable.
\item[ii)] There is a sequence $\{ E_n\}_{n\geq 1}$ of diagonal projections such that 
\beq\label{eq cond carac coro}
\sum_{n\geq 1} (P_n-E_n)^2\in\cJ^2\,,
\eeq
where the series  is assumed to converge in the WOT. 
\end{enumerate}
\end{cor}
\begin{proof}
 Notice that under our present assumptions, $\cJ$ is also a proper arithmetic mean closed operator ideal.
Assume that $A$ is $\cU_\cJ(\cH)$-diagonalizable and let $U\in\cU_\cJ(\cH)$ be such that $UAU^*=B\in\cD$. 
If we let $K=U-I \in \mathcal{J}$ and $E_n=UP_nU^*\in\cD$, for $n\in\N$, then
$$  P_ n - E_n=P_n - KP_nU^* - P_n K^* - P_n= - KP_n U^* - P_n K^*=- UP_n K^*  - K P_n \,,$$ 
where we used that $P_ n - E_n$ is self-adjoint. Hence,
\begin{align*}
 \sum_{n\geq 1}(P_n- E_n)^2 & = \sum_{n\geq 1}(KP_nU^* + P_n K^*) (UP_n K^*  + K P_n)\\  
& = KK^*+  K \sum_{n\geq 1}P_nU^* KP_n +  \left(\sum_{n\geq 1}P_nK^*U P_n \right)\,K^*+ \sum_{n\geq 1} P_n K^*KP_n\in\cJ^2
\end{align*}
where each series converges in the operator norm by the same arguments as in Eq. (\ref{series varias}), and where we have used that $\mathcal{J}$ and $\cJ^2$ are arithmetic mean closed (see Remark \ref{pinching properties}). Thus, $\sum_{n\geq 1}(P_n- E_n)^2 \in \mathcal{J}^2$.

\medskip

Conversely, note that 
\begin{align*}
\sum_{n\geq 1}(P_n-E_n)^2 & = \sum_{n\geq 1}(I-E_n)P_n(I-E_n)  +   \sum_{n\geq 1}E_n(I-P_n)E_n \\
& \geq \sum_{n\geq 1}(I-E_n)P_n(I-E_n),
\end{align*}
which implies that the last series belongs to $\cJ^2$ since operator ideals are hereditary. Then, using that $\SE{P}{n}$ are mutually orthogonal, we get that $|\sum_{n\geq 1}P_n(I-E_n)|^2 \in \cJ^2$. By the polar decomposition, this implies that
$
\sum_{n\geq 1}P_n(I-E_n) \in \cJ.
$
Using that the involution is WOT-continuous, we get $\sum_{n\geq 1}(I-E_n)P_n \in \cJ$.
Since $\cJ^2 \subseteq \cJ$, and 
$$
\sum_{n\geq 1}(P_n-E_n)^2 = \sum_{n\geq 1}P_n(I-E_n)  +   \sum_{n\geq 1}E_n(I-P_n),
$$
we obtain $\sum_{n\geq 1}E_n(I-P_n) \in \cJ$. The result now follows by Theorem \ref{main result}.
\end{proof}

\begin{exa}
In the particular case $\cJ=\fS_2(\cH)$, then $\cJ^2=\fS_1(\cH)$, which is a proper arithmetic mean closed ideal. Hence, 
 given
a (normal) diagonalizable operator $A$ with spectral projections $\{ P_n\}_{n=1}^N$ $(N\in\N$ or $N= \infty)$: $A$ is $\cU_{\fS_2(\cH)}(\cH)$-diagonalizable if and only if there exist diagonal projections $\{E_n\}_{n=1}^N$ such that
$$
\tr\left(\sum_{n=1}^N (P_n-E_n)^2\right)=\sum_{n=1}^N\|P_n-E_n\|_{\fS_2(\cH)}^2<\infty\,.
$$  
Indeed, if $N\in\N$, then this follows from Loreaux's Theorem \ref{teo1 L19 original}; and if $N=\infty$, then it follows from Corollary \ref{coro main result}.
\end{exa}

\begin{cor}\label{cor sobre desc indent}
Let $\cJ$ be a proper arithmetic mean closed  operator ideal and let $\{ E_n \}_{n\geq 1}$ and $\{ P_n\}_{n\geq 1}$ be two (infinite) decompositions of the identity. Then there exists $U\in\cU_\cJ(\cH)$ such that 
$U\,P_nU^*=E_n$, for all $n \in \N$, if and only if $[P_n:E_n]=0$ for $n\in\N$, and the conditions in Eq. \eqref{eq cond caract1} hold,
 where both series always converge weakly in $\cB(\cH)$.
\end{cor}
\begin{proof}
 Assume first that there exists $U\in\cU_\cJ(\cH)$ such that 
$U\,P_nU^*=E_n$, for $n\in\N$.  Then, by Proposition \ref{cond proj res diag} we have that $[P_n:E_n]=0$, for $n\in\N$. On the other hand, we can argue as in the proof of the second assertion in Theorem \ref{main result} and conclude that the conditions in Eq. \eqref{eq cond caract1} hold. 

Conversely, assume that the conditions in Eq. \eqref{eq cond caract1} hold. 
We can choose the orthonormal basis $\e=\SE{e}{n}$ in such a way that the family $\SE{E}{n}$ lies in its associated diagonal algebra $\cD$. Now we can argue as in the proof of the first assertion in Theorem \ref{main result}. If we follow that argument, we see that the family $\{E'_n\}_{n=0}^{n_1}$ actually coincides with $\{E_n\}_{n=0}^{n_1}$ with $E_0=\sum_{n>n_1}E_n$. From this last fact, we now see that there exists $U\in\cU_\cJ(\cH)$ such that $U\,P_nU^*=E_n$, for all $n\in\N $, as desired. 
\end{proof}

Let $\cJ$ be an operator ideal in $\cB(\cH)$ and let $\mathrm{e}=\{e_n\}_{n\geq 1}$ $\mathrm{f}=\{f_n\}_{n\geq 1}$ be two orthonormal bases of $\cH$. According to \cite[Remark 4.5]{BPW16} these bases are \textit{$\cJ$-equivalent} if there exists $W\in\cU_\cJ(\cH)$
such that $Wf_n=e_n$, for $n\in\N$.

\begin{rem}\label{rem bases J equiv} Let $\mathrm{e}=\{e_n\}_{n\geq 1}$ $\mathrm{f}=\{f_n\}_{n\geq 1}$ be two orthonormal bases of $\cH$.  In \cite[Problem 6.3]{BPW16} the authors ask for
necessary or sufficient conditions for the $\cJ$-equivalence of $\mathrm{e}$ and $\mathrm{f}$. Notice that Corollary \ref{cor sobre desc indent} provides with a \textit{projective} solution to this problem, 
in case $\cJ$ is an arithmetic mean closed ideal. Indeed, given $\mathrm{e}$ and $\mathrm{f}$ as before, we can consider the associated decompositions of the identity $E_n=e_n\otimes e_n$ and $P_n=f_n\otimes f_n$, for $n\in\N$. Then,  Corollary \ref{cor sobre desc indent}
characterizes the existence of $U\in\cU_\cJ(\cH)$ such that $UP_nU^*=E_n$, i.e. such that $Uf_n=\alpha_n\,e_n$ for some $\alpha_n\in\C$ with $|\alpha_n|=1$, for $n\in\N$. Nevertheless, this result does not seem to be satisfactory.  
\end{rem}

We now obtain necessary and sufficient conditions for $\cJ$-equivalence between orthonormal bases for an arbitrary proper operator ideal $\cJ$. In order to do this, we first notice the following: if $\mathrm{e}=\{e_n\}_{n\geq 1}$, $\mathrm{f}=\{f_n\}_{n\geq 1}$ and $\mathrm{g}=\{g_n\}_{n\geq 1}$ are orthonormal bases of $\cH$ then there exists a (unique) well defined $T\in \cB(\cH)$ such that $T(g_n)=e_n-f_n$, for $n\in\N$ (this last fact is equivalent to the assertion that $\{e_n-f_n\}_{n\geq 1}$ is a Bessel sequence in $\cH$, see \cite{Chris16}). Indeed, notice that the unitary operators $U_1,\,U_2\in\cU(\cH)$ given by $U_1(g_n)=e_n$, $U_2(g_n)=f_n$, for $n\in\N$, are such that $T=U_1-U_2$. 

\begin{cor}\label{cor jequiv bases}
Let $\cJ$ be a proper operator ideal  and let  $\mathrm{e}=\{e_n\}_{n\geq 1}$, $\mathrm{f}=\{f_n\}_{n\geq 1}$ be two orthonormal bases in $\cH$. Then, $\mathrm{e}$ and $\mathrm{f}$ are $\cJ$-equivalent if and only if for  every (equivalently, some) orthonormal basis $\mathrm{g}=\{g_n\}_{n\geq 1}$ of $\cH$ the operator $T\in \cB(\cH)$ given by $T(g_n)=e_n-f_n$, for $n\in\N$, verifies that $T\in\cJ$.
\end{cor}
\begin{proof}
Let $\mathrm{e}$, $\mathrm{f}$ and $\mathrm{g}$ be orthonormal bases. Consider 
the unitary operators $U_1,\,U_2\in\cU(\cH)$ given by $U_1(g_n)=e_n$, $U_2(g_n)=f_n$, for $n\in\N$; in this case we have that $T=U_1-U_2\in \cB(\cH)$. Notice that $\mathrm{e}$ and $\mathrm{f}$ are $\cJ$-equivalent if and only if there exists $W\in\cU_\cJ(\cH)$ such that $WU_2=U_1$. Since $U_1$ and $U_2$ are, in particular, partial isometries such that $N(U_1)=N(U_2)=\{0\}$ then, Proposition \ref{partial iso charac}
together with the previous remark show that $\mathrm{e}$ and $\mathrm{f}$ are $\cJ$-equivalent if and only if $U_1-U_2=T\in\cJ$.
\end{proof}

\begin{exas}\label{exa jequiv bases}
Let  $\mathrm{e}=\{e_n\}_{n\geq 1}$, $\mathrm{f}=\{f_n\}_{n\geq 1}$ be two orthonormal bases of $\cH$. 

\medskip

\noi $i)$. Let $\cJ=\fS_2(\cH)$, and notice that if $\mathrm{g}=\{g_n\}_{n\geq 1}$ is an orthonormal basis of $\cH$ and $T\in \cB(\cH)$ is such that $T(g_n)=e_n-f_n$, for $n\in\N$, then $T\in \fS_2(\cH)$ if and only if 
$\sum_{n\geq 1}\|Tg_n\|^2=\sum_{n\geq 1}\|e_n-f_n\|^2<\infty$. Hence, we recover from Corollary \ref{cor jequiv bases} the following well-known fact: $\mathrm{e}$ and $\mathrm{f}$ are $\fS_2(\cH)$-equivalent if and only if $\sum_{n\geq 1}\|e_n-f_n\|^2<\infty$.

\medskip

\noi $ii)$. Let $\cJ=\cK(\cH)$ and $\mathrm{g}=\{g_n\}_{n\geq 1}$ be an orthonormal basis of $\cH$, and let $T\in \cB(\cH)$ be such that $T(g_n)=e_n-f_n$, for $n\in\N$. For each $N\in\N$, let $Q_N$ denote the orthogonal projection onto $\text{Span}\{g_n\}_{n=1}^N$. Then, $T\in \cK(\cH)$ if and only if 
$\lim_{N\rightarrow \infty}\|T-TQ_N\|=0$. Hence, Corollary \ref{cor jequiv bases} shows that 
$\mathrm{e}$ and $\mathrm{f}$ are $\cK(\cH)$-equivalent if and only if for every $\varepsilon>0$, there exists $N\in\N$ such that for all $\alpha=(\alpha_n)_{n\geq 1}\in\ell^2(\N)$ we have that 
$$
\left\|\sum_{n>N}\alpha_n\,(e_n-f_n)\right\|\leq \varepsilon\, \|\alpha\|_{\ell^2}\,.
$$
\end{exas}

\section{On the structure of the set $\cD_\cJ^{sa}$}\label{structure of dj}

This section is devoted to study the way that the self-adjoint part of $\cD_\cJ$ sits inside the self-adjoint part of an operator ideal $\cJ$. In particular, we 
answer some questions raised in \cite{BPW16} on the structure of restricted unitary orbits.
Let $\cJ^{sa}$ be the self-adjoint part of the operator ideal $\cJ$. The self-adjoint part of $\cD_\cJ$ is given by
$$
\cD_{\cJ}^{sa}:=\cD_{\cJ} \cap \cJ^{sa}=\{ \, VDV^*\, : \,   D=D^* \in \cD \cap \cJ, \, V \in \cU_{\cJ}(\cH)  \,\}.
$$
As it is shown in \cite{BPW16}, the inclusion $\cD_\cJ^{sa}\subsetneq\cJ^{sa}$ is proper in the general situation. The authors of that work also describe an example in which $\cD_\cJ^{sa}$ is not a linear subspace, but the linear span satisfies
$$
\text{Span}_\R(\cD_\cJ^{sa})=\cJ^{sa}.
$$  
In fact, their example is the ideal $\cJ=B(\cH)$. This naturally led them to consider (see \cite[Question 3.3]{BPW16}) if the set $\cD_{\cJ}^{sa}$ is a proper subset of $\cJ^{sa}$
for every operator ideal $\cJ$; similarly, they asked about the relation between $\cJ^{sa}$ and the real linear span of $\cD_{\cJ}^{sa}$ for proper operator ideals. Now we show that in the case  which  
$\cJ\neq \cF(\cH)$ is a proper operator ideal, then $\cD_\cJ^{sa}$ is not a linear space. Moreover, if we assume further that $\cJ\neq \cJ^2$, then we prove that $\text{Span}_\R(\cD_\cJ^{sa})\neq\cJ^{sa}$.

\begin{teo}\label{teo una pregunta1}
Let $\cJ \neq \cF(\cH) $ be a proper operator ideal. Then, $\cD_\cJ^{sa}$ is not a linear space. 
\end{teo}
\begin{proof}
Our proof is based on the construction of two operators $X,\,Y\in \cD_\cJ^{sa}$ such that 
$X+Y\notin \cD_\cJ^{sa}$. The construction has a block-diagonal structure with blocks of size 2; hence we begin with some computations related with $2\times 2$ matrices.

Consider the $2\times 2$ (real) unitary matrix
$$
U(\theta)=\begin{pmatrix} \cos (\theta) & - \sin (\theta) \\
\sin (\theta) & \cos (\theta) \end{pmatrix}, \, \, \, \, \,  \theta\in [0,\pi/2]\,.
$$
Then, the eigenvalues of this normal matrix are $e^{i\,\theta}$, $e^{-i\,\theta}$.
Hence, we can write 
$$
U(\theta)=I_2 + (U(\theta)-I_2), \, \, \, \, \,  s(U(\theta)-I_2)=(|1-e^{i\,\theta}|\coma |1-e^{-i\,\theta}|)\,.
$$
Here $I_2$ denotes the $2 \times 2$ identity matrix. Notice that in this case, given $0<t<\sqrt{2-2/\sqrt 2 }$, then there exists a unique $0<\theta<\pi/4$ such that 
$$
|1-e^{i\,\theta}|=|1-e^{-i\,\theta}|=t , \peso{given by} \theta=\arccos\left(\frac{2-t^2}{2}\right)\,.
$$
Now, consider two vectors $x,\,y\in \C^2$ such that $\|x\|=\|y\|=1$ and $0<\langle x\coma y\rangle <1$. The positive definite $2\times 2$ matrix
$xx^*+yy^*$ has two different eigenvalues $1+\langle x\coma y\rangle>0$ and $1-\langle x\coma y\rangle>0$, such that $$v_1= \frac{x+y}{\|x+y\|}$$ is a (well defined) unit norm eigenvector associated with the eigenvalue $1+\langle x\coma y\rangle$.
Finally, notice that if we consider the (real) unitaries $U(\theta)$ and $U(-\theta)$ for some $\theta\in (0,\pi/4)$ and the canonical basis of $\C^2$ given by $\{f_1,f_2\}$ then, if we set
$$
x_\theta=U(\theta)\,f_1=\begin{pmatrix} \cos(\theta) \\ \sin(\theta)\end{pmatrix} \py 
y_\theta=U(-\theta)\,f_2=\begin{pmatrix} \sin(\theta) \\ \cos(\theta)\end{pmatrix} 
$$
we have that $\|x_\theta\|=\|y_\theta\|=1$ and $0<\langle x_\theta\coma y_\theta\rangle =\sin(2\theta)< 1$. Hence, the previous computations show that the positive semidefinite $2\times 2$ matrix 
$x_\theta\,x_\theta^*+y_\theta\,y_\theta^*$ has distinct eigenvalues given by 
$1+ \sin(2\theta)$ and $1-\sin(2\theta)$, and unit norm eigenvectors,
$$
v=\frac{x_\theta+y_\theta}{\|x_\theta+y_\theta\|}=\frac{f_1+f_2}{\sqrt 2} \py  w=\frac{f_1-f_2}{\sqrt 2} .
$$
Notice that once we have computed the vector $v$, we can take $w$ to be any unit norm vector that is orthogonal to $v$, since the matrix is positive semi-definite.

We now go back to the separable infinite dimensional Hilbert space $\cH$ together with its orthonormal basis $\e=\{e_n\}_{n\geq 1}$. Consider $\cJ\neq \cB(\cH)$ and $\cJ\neq \cF(\cH)$. Then, the characteristic set $\Sigma(\cJ)\subset c_0$ and we can consider a sequence $\alpha=\{\alpha_j\}_{j \geq 1}\in \Sigma(\cJ)$ such that 
\ben
\item $0<\alpha_{j+1}<\alpha_j<\sqrt{2-2/\sqrt 2 }$, for $j\geq 1$;
\item $\alpha_j\,(1\pm \sin(2\,\arccos(\frac{2-\alpha_j^2}{2})))\neq \alpha_k\,(1\pm \sin(2\,\arccos(\frac{2-\alpha_k^2}{2})))$, for $j\neq k$.
\een The fact that 
$$
\lim_{x\rightarrow 0^+} x\,\left(1\pm \sin\left(2\,\arccos\left(\frac{2-x^2}{2}\right)\right)\right)=0
$$
implies that it is always possible to choose the sequence $\{\alpha_j\}_{j \geq 1}$ as above.
For each $j \geq 1$, we let $\theta_j=\arccos\left(\frac{2-\alpha_j^2}{2}\right)$ (notice that $\theta_j$ is a increasing function of $\alpha_j$). Hence, 
$$
\theta_j\in (0,\pi/4) \ \ , \ \ |1-e^{\pm i\theta_j}|=\alpha_j  \py
\lim_{j\rightarrow \infty} \theta_j=0\,.
$$ 
We set
$$
V=\bigoplus_{j\geq 1} \, U(\theta_j)\in \cU_\cJ(\cH)\,,
$$
where each copy $U(\theta_j)$ is acting on the subspace
 $\cH_j=\text{Span}\{e_{1+2\,(j-1)}\coma e_{2j} \}$, for $j\geq 1$.
Indeed,
$$
V=I+\bigoplus_{j\geq 1} \,(U(\theta_j)-I_2)=I+K, \, \, \, \, \,  s(K)=\{ \, ( \alpha_j\coma \alpha_j)\, \}_{j\geq 1}\in\Sigma(\cJ)\,, 
$$
since $|1-e^{\pm\, i\, \theta_j}|=\alpha_j$, by construction of $\theta_j$. Notice that we have used that characteristic sets are invariant under  ampliations (see Section \ref{Prelim}) to conclude that $s(K)\in\Sigma(\cJ)$. Hence, $K\in\cJ$.
Similarly, we consider 
$$W=\bigoplus_{j\geq 1} \, U(-\theta_j)\in \cU_\cJ(\cH)\,. $$
Consider $A,\,B\in \cD$ the self-adjoint diagonal operators given by
$$
A=\sum_{j\geq 1} \alpha_{j}\, P_{1+2(j-1)}= \bigoplus_{j\geq 1}\, \alpha_{j}\, f_1 f_1^*
\py B=\sum_{j \geq 1} \alpha_{j}\, P_{2j}=\bigoplus_{j\geq 1}\, \alpha_{j}\, f_2 f_2^* \,,
$$where $P_j$ denotes the orthogonal projection onto $\C\,e_j$, for $j\geq 1$, and the direct sums
of these $2\times 2$ blocks $\alpha_j\,f_k  f_k^*$, $k=1,2$, are considered as before (see the definitions of $V$ and $W$).
Now we define $X,\,Y\in \cD_\cJ^{sa}$ given by 
$$
X=V\,A\,V^* =\bigoplus_{j\geq 1}\, \alpha_{j}\, x_{\theta_j}x_{\theta_j}^*\py 
Y=W\,B\,W^* =\bigoplus_{j\geq 1}\, \alpha_{j}\, y_{\theta_j}y_{\theta_j}^*\,,
$$
where we have used the block diagonal structure for $A,\,B,\,V$ and $W$.
Finally, notice that 
$$
X+Y=\bigoplus_{j\geq 1}\, \alpha_{j}\, (x_{\theta_j}x_{\theta_j}^* +y_{\theta_j}y_{\theta_j}^*)\in \cJ^{sa}\,.
$$
We can further use the block diagonal structure of $X+Y$ to conclude that the eigenvalues of 
$X+Y$ are given by the strictly positive numbers
$$
\alpha_j\,(1+ \sin(2\theta_j))\py \alpha_j \,(1-\sin(2\theta_j)), \, \, \, \, \,  j \geq 1\,
$$
with corresponding eigenvectors
\beq \label{eq eigenvec1}
v_j=\frac{e_{1+2(j-1)} + e_{2j}}{\sqrt 2}\py 
w_j=\frac{e_{1+2(j-1)} - e_{2j}}{\sqrt 2}, \, \, \, \, \,  j \geq 1\,.
\eeq 
Then, by construction of $\{\alpha_j\}_{j \geq 1}$, the eigenvalues of $X+Y$ (as described above) are all simple. Therefore, we obtain the compact self-adjoint operator $X+Y$ with simple eigenvalues, whose eigenvectors are given by the Eq. \eqref{eq eigenvec1}. Now we apply Proposition \ref{pro 35 BPW16}: thus, we consider an eigenvector of $X+Y$, say 
$v_j$ (the argument with $w_j$ is completely analogous) and a vector $e_k$ in the orthonormal basis $\rm e$. Notice that in case $k\notin\{ 1+2(j-1)\coma 2j\}$, then
$\langle v_j\coma e_k\rangle =0$ and therefore, $\|v_j - \beta\,e_k\|=\sqrt 2$, for every $|\beta|=1$. In case $k=2j$ and $|\beta|=1$, then
$$
\| v_j - \beta\,e_k\|=\left\| \frac{1}{\sqrt 2}\,e_{1+2(j-1)} - \left(\frac{1}{\sqrt 2}  -  \beta \right)\,e_{2j}  \right\|\geq \frac{1}{\sqrt 2}\,.
$$
Similarly, if $k=1+2(j-1)$, then we have that $\| v_j - \beta\,e_k\|\geq 1/\sqrt 2$. 
Thus, Proposition \ref{pro 35 BPW16} implies that $X+Y\notin \cD_\cJ$.
\end{proof}

\begin{rem}
The above result implies that  $\cD_\cJ^{sa}\neq \cJ^{sa}$, when $\cJ \neq \cF(\cH)$; this fact was already proved in \cite[Corollary 3.7]{BPW16}.
\end{rem}

\begin{teo}\label{teo una pregunta2}
Let $\cJ$ be an operator ideal such that $\cJ^2\neq \cJ$.
Then, $\mathrm{Span}_\R(\cD_\cJ^{sa})\neq \cJ^{sa}$.
\end{teo}
\begin{proof}
Our argument is simpler if we let the orthonormal basis of $\cH$ be described as
$\e=\{e_j\}_{j\in\Z}$ (i.e. we use the integers to describe the elements of the fixed orthonormal basis 
of $\cH)$. Assume that 
$A\in \cD_\cJ$ and consider $V=1+K\in U_\cJ(\cH)$, for some $K\in\cJ$ and $D\in\cD$ such that 
$$A=V^*\,D\,V= (1+K)^*\,D\,(1+K)= D+K^*D+DK+K^*DK\,.$$ In this case, $D=V\,A\,V^*\in\cJ$, too.
Let $J\subset \Z$ be any subset  and let
$P_J$ be the orthogonal projection onto the closure of the subspace spanned by 
$\{e_j\ : \ j\in J\}$. Then, we consider the $(1,2)$ anti-diagonal block of the matrix representation 
of $A$ with respect to decomposition of the identity $\{P_J, 1-P_J\}$, that is
\begin{eqnarray*}
P_J \,A\,(1-P_J)&=&P_J \,(D+K^*D+DK+K^*DK)\,(1-P_J)\\ &=& P_J \,(K^*D+DK+K^*DK)\,(1-P_J )\in \cJ^2\, . \end{eqnarray*}
Here we have used that $P_J \,D \,(1-P_J )=0$, since $D\in\cD$ is a diagonal operator with respect to the orthonormal basis $\rm e$.

Hence, if $A_1,\ldots,A_n\in \cD_\cJ$ and $P_J$ is as before, then
$$
P_J \,(A_1+\ldots+A_n)\,(I-P_J)=P_J \,A_1\,(I-P_J )+\ldots+P_J \,A_n\,(I-P_J )\in\cJ^2\,.
$$
Now, let $X\in \cB(\cH\ominus \cK\coma \cK)$ be such that $s(X)\in \Sigma(\cJ)\setminus \Sigma(\cJ^2)$, where $\cK\subset \cH$ is the closure of the subspace spanned by $\{e_j\}_{j \geq 1}$. Then, we consider the operator 
$$
B=\begin{pmatrix} 
0 & X \\ X^* & 0
\end{pmatrix}.  
$$
It is well known that in this case 
$$s(B)=D_2(s(X))=(s_1(X),s_1(X),s_2(X),s_2(X),\ldots)\in\Sigma(\cJ)\,,$$ 
 so  $B\in \cJ^{sa}$ using that characteristic sets are invariant by ampliations.
Now we choose $J=\N\subset \Z$. Therefore, $P_J=P_\cK$ is the orthogonal projection onto $\cK$, so
$$
P_J\,B\,(1-P_J)=\begin{pmatrix} 
0 & X \\ 0 & 0
\end{pmatrix} \notin \cJ^2,
$$ since $s(P_J\,B\,(I-P_J))=s(X)\in \Sigma(\cJ)\setminus \Sigma(\cJ^2)$. Hence, $B$ cannot lie in the linear span of $\cD_\cJ^{sa}$.
\end{proof}

For our next result on the structure of restricted orbits, we 
recall
\cite[Prop. 3.1]{BPW16} and 
\cite[Problem 6.6]{BPW16}.

\begin{pro}\label{pro BPW16 sobre diag}
If a normal compact operator $X \in \cJ \neq \cB(\cH)$ with spectral multiplicities equal to one is $\cU_{\cJ}(\cH)$-diagonalizable to a diagonal operator $D\in\cD$ with respect to the fixed basis $\e=\{ e_n\}_{n \geq 1}$, then $D$ is unique up to a finite permutation. Furthermore, 
$X$ can be $\cU_{\cJ}(\cH)$-diagonalized to every finite permutation of $D$. 
\end{pro}

In what follows we extend the analysis of Proposition \ref{pro BPW16 sobre diag}
to the general context of (normal) diagonalizable operators. We fix a  diagonalizable operator
$$
A=\sum_{n=1}^N \la_n \, P_n \, ,\, \, \, \, \, \, \, \, \,  N\in\N\ \ \ \ \text{or} \ \ \ \ N= \infty,
$$   
where the series converges strongly when $N=\infty$. We assume that $A\in\cD_\cJ$ for some proper operator ideal $\cJ$,
and consider its  restricted unitary orbit, i.e.
$$
\cO_\cJ(A)=\{ VAV^* \, : \, V \in \cU_{\cJ}(\cH) \}.
$$
In order to tackle
the previous 
problem,  we consider the set
$$ 
\cO_\cJ(A)\cap\cD=\{ B \in \cD \, : \, B=VAV^* \text{ for some }  V\in\cU_\cJ(\cH) \}\,.
$$
We will show that $\cO_\cJ(A)\cap\cD$ can also be  described as an orbit 
under the action of the finite permutations. Indeed,
let $U\in\cU_\cJ(\cH)$ be such that $$B=UAU^*=\sum_{n=1}^N \la_n \, E_n \in\cO_\cJ(A)\cap\cD\,,$$
where $E_n\in\cD$, for $1\leq n\leq N$, are such that $\{E_n\}_{n=1}^N$ is a decomposition of the identity.
Our next result is a consequence of the characterization obtained in Theorems \ref{teo1 L19}  and \ref{main result} (see also Remark \ref{additional rem}).

\begin{teo}\label{teo caract de las posibles diagonalizaciones}
With the previous notation, 
\beq \label{eq ident orbitadora1}
\cO_\cJ(A)\cap\cD=\{U_\sigma B U_\sigma^*\ : \ \sigma \text{ is a finite permutation } \}
\eeq where $U_\sigma\in\cU_\cJ(\cH)$ denotes the unitary operator induced by the finite permutation $\sigma$.
\end{teo}
\begin{proof}
 The inclusion of the set to the right into the set to the left in Eq. \eqref{eq ident orbitadora1} is clear, since $U_\sigma-I$ is a finite rank operator, for each finite permutation $\sigma$.

To show the other inclusion notice that, with the previous notation, $\{E_n\}_{n=1}^N$ induces a partition $\cP=\{I_n\}_{n=1}^N$ of $\N$ in such a way that
$E_n$ is the orthogonal projection onto the closure of subspace spanned by $\{e_i:\ i\in I_n\}$.

Take $B'\in \cO_\cJ(A)\cap\cD$. Then,  $B'=\sum_{n=1}^N \la_n \, E'_n \in\cD$ for some decomposition of the identity $\{E'_n\}_{n=1}^N$ in $\cD$. In this case, based on the theory of multiplicity of eigenvalues, there exists a permutation $\tau:\N\rightarrow \N$ such that $E'_n$ is the orthogonal projection onto the closure of the subspace spanned by $\{e_{\tau(i)}\ : \ i\in I_n\}$, for $1\leq n\leq N$.

Consider 
$$
\cM(\tau,\cP)=\{i \geq 1\ : \ i\in I_n\ , 1\leq n\leq N \ \text{ and } \ \tau(i)\notin I_n  \}\,.
$$ 
We now show that $\cM(\tau,\cP)$ is a finite set by way of contradiction. Thus, we assume that $\cM(\tau,\cP)$ is an infinite set and consider the following two cases:

Case 1: $\{1\leq n\leq N\ :  \ \exists i\in I_n \ , \ \tau(i)\notin I_n\}$ is a finite set. In this case, there exists $1\leq n\leq N$ such that $\cN=\{i\in I_n \ , \ \tau(i)\notin I_n\}$ is an infinite set.
Then, by hypothesis, $P_n-E_n\in\cJ$ and $[P_n:E_n]=0$. In particular, 
$\lim_{i\in\cN,\,i\rightarrow \infty}\|(E_n-P_n)(e_i)\|=0$. Hence there exists $i_0\in\cN$ such that 
for $i \geq i_0$ we have that $\|e_i-P_n(e_i)\|\leq 1/2$. Notice that in this case $\|e_i-(I-P_m)(e_i)\|\leq 1/2$, for $i\in\cN$, $i\geq i_0$ and $m\neq n$. Now fix $i\in\cN$, $i\geq i_0$ and let 
$m\neq n$ be such that $E'_me_i=e_i$, so that 
$$
\|e_i-E'_m(I-P_m)(e_i)\|=\|E'_m (e_i- (I-P_m)(e_i))\|\leq 1/2.
$$
Therefore,
$$
\|E'_m(I-P_m)(e_i)\|\geq 1/2\,.
$$
The previous inequality shows that $\|\sum_{r=1}^N E'_r(I-P_r)(e_i)\|\geq \|E'_m(I-P_m)(e_i)\|\geq 1/2$, for $i\in\cN$, $i\geq i_0$. Since the set of such indexes is infinite, we see that $\sum_{r=1}^N E'_r(I-P_r)$ is not a compact operator, which contradicts the fact that $B'\in\cO_\cJ(A)\cap \cD$ by 
Theorem \ref{teo1 L19} if $N\in \N$ or by item $iv)$ in Remark \ref{additional rem} if $N=\infty$.

Case 2: $\{1\leq n\leq N\ :  \ \exists i\in I_n \ , \ \tau(i)\notin I_n\}$ is an infinite set (hence $N=\infty$). In this case we notice that by Lemma \ref{finite diff} we have that $\lim_{n\rightarrow \infty}\|E_n-P_n\|=0$. 
We argue as before: let $n_0\geq 1$ be such that $\|E_n-P_n\|\leq 1/2$ for $n \geq n_0$. Then, the set
$\cN=\{i\in\N\ : \ \exists n\geq n_0\ , \ i\in I_n \ , \ \tau(i)\notin I_n\}$ is infinite. For each 
$i\in\cN$ we have that $\|\sum_{r\geq 1} E'_r(I-P_r) (e_i)\|\geq 1/2$, as before. Hence, in this case 
$\sum_{r\geq 1} E'_r(I-P_r)$ is not a compact operator, which again contradicts the fact that $B'\in\cO_\cJ(A)\cap \cD$, by item $iv)$ in Remark \ref{additional rem}. 

Hence $\cM(\tau,\cP)$ is a finite set. Now, it is a combinatorial exercise to show that there exists a finite permutation $\sigma:\N\rightarrow \N$ such that $\sigma(I_n)=\tau(I_n)$, for $1\leq n\leq N$ 
(see Remark \ref{hay perm finit} below).
\end{proof}

\begin{rem}\label{hay perm finit}
We sketch a brief argument to prove the last claim in the previous proof. We apply induction on the number of elements $\#(\cM(\tau,\cP))$. Fix $i\in \cM(\tau,\cP)$. Consider the following cases:

Case 1: there exists $1\leq k_0\in\N$ such that $\tau^{k_0}(i)=i$. If we consider the finite cycle $\mu=(i \ \tau(i) \ \ldots \ \tau^{k_0-1}(i))$ then
$\mu^{-1}\circ \tau$ is a permutation such that $\mu^{-1}\circ \tau(\tau^r(i))=\tau^r(i)$, for $r\in\Z$
 and $\mu^{-1}\circ \tau(j)=j$, $j\neq \tau^r(i)$ and $r\in\Z$. Hence,
$\#(\cM(\mu^{-1}\circ \tau,\cP))\leq  \#(\cM(\tau,\cP))-2$ (since there are at least two jumps from different $I_m$'s in the cycle, that we have erased). 
In this case we can apply our inductive hypothesis to $\mu^{-1}\circ \tau $ and obtain a finite permutation $\sigma'$ such that $\mu^{-1}\circ \tau(I_n)=\sigma'(I_n)$ for $1\leq n\leq N$, from which it follows that $\sigma=\mu\circ \sigma'$ has the desired properties.

Case 2: $\tau^k(i)\neq i$ for $k\in\Z$. In this case there exists $n_0\in\N$ and $1\leq k'\in\N$ such that $\tau^k(i)\in I_{n_0}$, for $k\geq k'$. Let $k'\leq k_0=2\,\alpha+1$, for some $\alpha\in\N$.  Similarly, there exists an index $r'\leq 0$
such that $\tau^k(i)\in I_{m_0}$, for some (unique) $m_0\in\N$, for every $k\leq r'$; let $r_0=-2\,\beta\leq r'$ for some $\beta\in\N$. We now define the finite cycle $\rho:\N\rightarrow \N$ given by 
$\rho=( \tau^{r_0}(i) \ \tau^{r_0+1}(i) \ \ldots \ \tau^{k_0}(i))$. 
 Notice that in this case $\cM(\rho\circ \tau,\cP)\leq \cM(\tau,\cP)+1$ (since we are adding a jump
from $I_{n_0}$ to (the possibly different set) $I_{m_0}$ at the value $\rho\circ \tau(\tau^{k_0-1}(i))=\tau^{r_0}(i)$). On the other hand, it is easy to check that $(\rho\circ \tau)^{\alpha}=\tau^{r_0}(i)$ and that $(\rho\circ \tau)^{\alpha+\beta}=\tau^{0}(i)=i$. We now construct, as in case 1, the finite cycle $\mu=(i \ \  (\rho\circ \tau)(i)\  \ldots\  (\rho\circ \tau)^{\alpha+\beta-1}(i))$ so that 
$\cM(\mu^{-1}\circ (\rho\circ \tau),\cP)\leq \cM(\rho\circ \tau,\cP)-2\leq \cM(\tau,\cP)-1$.
In this case we can apply our inductive hypothesis to $\mu^{-1}\circ \rho\circ \tau $ and obtain a finite permutation $\sigma'$ such that $\mu^{-1}\circ \rho\circ \tau(I_n)=\sigma'(I_n)$ for $1\leq n\leq N$, from which it follows that $\sigma=\rho^{-1}\circ \mu\circ \sigma'$ has the desired properties.
\end{rem}

\begin{rem}
Consider the following equivalence relation in $\ell^\infty(\N)$: given $\alpha=\{\alpha_n\}_{n\geq 1}$,
$\beta=\{\beta_n\}_{n\geq 1}\in \ell^\infty(\N)$ then $\alpha\approx \beta$ if and only if there exists $U\in\cU(\cH)$ such that $D_\alpha=UD_\beta\,U^*$, where $D_\alpha,\,D_\beta\in\cD$ are the diagonal operators induced by the sequences $\alpha$ and $\beta$, respectively. It is straightforward to check that 
$\alpha\approx \beta$ if and only if $\cC=\{\alpha_n\ : \ n\in\N\}=\{\beta_n\ : \ n\in\N\}\subset\C$ and for every $\gamma\in\cC$ we have that $\#(\{n\in\N \ : \ \alpha_n=\gamma\})=\#(\{n\in\N \ : \ \beta_n=\gamma\})\in [0,\infty]$.

Given $\alpha\in\ell_\infty(\N)$ let $[\alpha]$ denote the equivalence class of $\alpha$ with respect to $\approx$. We can further consider the equivalence relation in $[\alpha]$ given by 
$\gamma\equiv\delta$ if there exists a finite permutation $\sigma$ such that $\gamma_n=\delta_{\sigma(n)}$, for $n\in\N$. We further denote by $[[\gamma]]$ the equivalence class of $\gamma\in[\alpha]$ with respect to $\equiv$.

Given a proper operator ideal $\cJ$ then Theorem \ref{teo caract de las posibles diagonalizaciones} allow us to describe the set $\cD_{\cJ, \,\cB(\cH)}$ of restricted diagonalizable operators as follows:
$$
\cD_{\cJ, \,\cB(\cH)}=\bigcup_{[\alpha]\in \ell_\infty(\N)/\approx} \ \bigcup_{[[\gamma]]\in [\alpha]/\equiv} \cO_{\cJ}(D_\gamma)
$$
where all unions are disjoint. 
\end{rem}

\subsection*{Acknowledgment}
This research was partially supported by  CONICET (PIP 2016 0525/ PIP 0152 CO), ANPCyT (2015 1505/ 2017 0883) and FCE-UNLP (11X829).

\bigskip

\noi (Eduardo Chiumiento) Departamento de  Matem\'atica \& Centro de Matem\'atica La Plata, FCE-UNLP, Calles 50 y 115, 
(1900) La Plata, Argentina  
and 
Instituto Argentino de Matem\'a\-tica, `Alberto P. Calder\'on', CONICET, Saavedra 15 3er. piso,
(1083) Buenos Aires, Argentina.

\noi e-mail: {\sf eduardo@mate.unlp.edu.ar}

\medskip

\noi (Pedro Massey) Departamento de  Matem\'atica \& Centro de Matem\'atica La Plata, FCE-UNLP, Calles 50 y 115, 
(1900) La Plata, Argentina  
and 
Instituto Argentino de Matem\'atica, `Alberto P. Calder\'on', CONICET, Saavedra 15 3er. piso,
(1083) Buenos Aires, Argentina.

\noi e-mail: {\sf massey@mate.unlp.edu.ar}

\end{document}